%% file: RCD_June_NEW.tex
\documentclass[final,leqno]{siamltex}

\input{commands}

\title{Robust Block Coordinate Descent \\ \vspace{0.5cm}\textnormal{M\MakeLowercase{ay} 8, 2015}}
\author{Kimon Fountoulakis and Rachael Tappenden}

\author{
        Kimon~Fountoulakis\thanks{K. Fountoulakis is with the School of Mathematics and Maxwell Institute, The University of Edinburgh, Edinburgh,
Mayfield Road, Edinburgh EH9 3JZ, United Kingdom e-mail: K.Fountoulakis@sms.ed.ac.uk.} \and
        Rachael~Tappenden\thanks{R. Tappenden is with the School of Mathematics and Maxwell Institute, The University of Edinburgh, Edinburgh,
Mayfield Road, Edinburgh EH9 3JZ, United Kingdom e-mail: r.tappenden@ed.ac.uk.}
}

\begin{document}

\maketitle
\begin{abstract}
In this paper we present a novel \textit{randomized} block coordinate descent method for the minimization of a convex composite objective function.  The method uses (approximate) partial second-order (curvature) information, so that the algorithm performance is more robust when applied to highly nonseparable or ill conditioned problems. We call the method Robust Coordinate Descent (RCD). At each iteration of RCD, a block of coordinates is sampled randomly, a quadratic model is formed about that block and the model is minimized \emph{approximately/inexactly} to determine the search direction. An inexpensive line search is then employed to ensure a monotonic decrease in the objective function and acceptance of large step sizes. We prove global convergence of the RCD algorithm, and we also present several results on the local convergence of RCD for strongly convex functions. Finally, we present numerical results on large-scale problems to demonstrate the practical performance of the method.
\end{abstract}


\section{Introduction}
In this work we are interested in solving the following convex composite optimization problem
\begin{equation}
\label{Def_F}
  \min_{x \in \R^N}  F(x) \eqdef f(x) + \Psi(x),
\end{equation}
where $f(x)$ is a smooth convex function and $\Psi(x)$ is a (possibly) nonsmooth, block separable, extended real valued convex function (this will be defined precisely in Section \ref{S_Psi}).
Problems of the form of \eqref{Def_F} arise in many important scientific fields, and applications include machine learning \cite{yuanho}, regression \cite{IEEEhowto:Tibshirani} and compressed sensing \cite{IEEEhowto:CandesRombergTao,Candes06,Donoho06}. Often the term $f(x)$ is a data fidelity term, and the term $\Psi(x)$ represents some kind of regularization.

Frequently, problems of the form of \eqref{Def_F} are large-scale problems, i.e., the size of $N$ is of the order of a million or a billion. Large-scale problems impose restrictions on the types of methods that can be employed for the solution of \eqref{Def_F}. In particular, the methods should have low per iteration computational cost, otherwise completing even a single iteration of the method might require unreasonable time. The methods must also rely only on simple operations such as inner products or matrix vector products, and ideally, they should offer fast progress towards optimality.

First order methods, and in particular randomized coordinate descent methods, have found great success in this area because they can take advantage of the underlying problem structure (separability and block structure), and satisfy the requirements of low computational cost and low storage requirements. For example, in \cite{petermartin} the authors show that their randomized coordinate descent method was able to solve sparse problems with millions of variables in a reasonable amount of time.

Unfortunately, randomized coordinate descent methods have two significant drawbacks. First, due to its coordinate nature, it is efficient only on problems with high degree of separability, and performance suffers when there is a high dependency between variables. Second, as a first-order method, coordinate descent methods do not usually capture essential curvature information of the problem and have been shown to struggle on complicated sparse problems \cite{l1regSCfg}.

The purpose of this work is to overcome these drawbacks by equipping a randomized block coordinate descent method with approximate partial second-order information. In particular, at every iteration of RCD the direction is obtained by solving \textit{approximately} a block piece-wise quadratic model, where the model includes a matrix representing approximate second order information. Then, a line search is employed in order to guarantee a monotonic decrease of the objective function.
%Note that our method is inspired by the work of Byrd, Nocedal and Oztoprak in \cite{sqa}.

RCD randomly selects a block of coordinates at every iteration, which is inexpensive. Although the per iteration computational cost of the method may be higher than other randomized coordinate descent methods, we show that in practise the method is more robust and the total number of iterations decreases. In particular we show that RCD is able to solve difficult problems, on which other coordinate descent method may struggle. RCD uses an inexact search direction, (the termination condition for the block piecewise quadratic subproblem is inspired by \cite{sqa}), coupled with a line search to ensure a monotonic decrease in the function values, and we prove global convergence of RCD and study its local convergence properties.

\subsection{Literature review}

Coordinate descent methods are some of the oldest iterative methods, and they are often better known in the literature under various names such as Jacobi methods, Gauss-Seidel methods, among others. It has been observed that these methods suffer from poor practical performance, particularly on ill-conditioned problems. However, as we enter the era of big data, coordinate descent methods are coming back into favour, because of their ability to provide approximate solutions of some realistic very large/huge scale problems in a reasonable amount of time.

Currently, randomized coordinate descent methods include that of Richt\`{a}rik and Tak\`{a}\v{c} \cite{petermartin}, where the method can be applied to unconstrained convex composite optimization problems of the form \eqref{Def_F}. The algorithm is supported by theoretical convergence guarantees in the form of high probability iteration complexity results, and \cite{petermartin} also reports very impressive practical performance on highly separable large scale problems. The work has also been extended to the parallel case \cite{Richtarik12}, to include acceleration techniques \cite{Fercoq13}, and to include the use of inexact updates \cite{Tappenden13}.

Other important works on randomized coordinate descent methods include methods for huge-scale problems \cite{Nesterov12}, work in \cite{Lu13} that improves the complexity analysis of \cite{Richtarik12}, coordinate descent methods for group lasso \cite{Qin13,Simon12} and general regularizers \cite{ShalevSchwartz13,Wright12} and coordinate descent for constrained optimization problems \cite{Necoara14}.

%In this paper the authors give a complete analysis of a general randomized coordinate
%descent for composite convex and possibly non-smooth objective functions. Although the analysis in \cite{petermartin} covers the case
%of block coordinate descent, it does not cover incorporation of approximate partial second-order information in a practical way. For example,
%the directions are calculated by the \textit{exact} minimization of piece-wise quadratic models where the model has to be an \textit{over-approximation}
%of the function to be minimized. Essentially exact minimization of the piece-wise quadratic model is practical only in the case that single coordinate updates are performed.

%Recently, there has been much interest with the incorporation of inexact information in these methods, as a way of making the algorithms more practical. Several papers, such as ???,  study the case where there is some inexactness in the gradient. Other works study the case where an inexact update is used. For example, in ? the authors extend the work of ? to include the case where \emph{inexact} updates are used. In?

%First order methods have found great success in this area

Unfortunately, on ill-conditioned problems, or problems that are highly nonseparable, first order methods can display very poor practical performance, and this has prompted the study of methods that employ second order information. To this end, recently there has been a flurry of research on Newton-type methods for problems of the general form \eqref{Def_F}, or a special case where $\Psi(x)=\|x\|_1$. For example, Karimi and Vavasis \cite{Karimi14} have developed a proximal quasi-Newton method for $l_1$-regularized least squares problems, Lee, Sun and Saunders \cite{Lee12,Lee13} have proposed a family of Newton-type methods for solving problems of the form \eqref{Def_F} and Scheinberg and Tang \cite{Scheinberg13} present iteration complexity results for a proximal Newton-type method. Moreover, the authors in \cite{sqa} extended standard inexact Newton-type methods to the case of minimization of a composite objective involving a smooth convex term plus an $l_1$-regularizer term.
Finally, there exists parallel deterministic \cite{facchinei14} and sequential active set \cite{desantis14} block coordinate descent methods, where the authors incorporate some block
second-order information in the algorithmic process.

We believe that the works in \cite{sqa,facchinei14} can lead the way to allow general \textit{randomized} coordinate descent to practically incorporate second-order information and in this paper we propose a method which combines the ideas in \cite{sqa,facchinei14,petermartin}.

\subsection{Core ideas and Major Contributions}
In this section we list several of the core ideas and major contributions of this work on randomized block coordinate descent methods. The first two points briefly describe the idea of incorporating (approximate) partial second-order (curvature) information (which
have been also presented in a similar way in \cite{facchinei14}), whilst the last three are contributions of this paper.
\begin{enumerate}
 \item \textbf{Incorporation of some second order information.} RCD uses a quadratic model to determine the search direction, which incorporates a user defined positive definite matrix $H\ii(x_k)$. If $H\ii(x_k)$ approximates the Hessian, then second order information is incorporated into the quadratic model, and the search direction obtained by minimizing the model is an approximate Newton-type direction. We stress that $H\ii(x)$ can change at every iteration and this is an advantage over the method in \cite{petermartin} where the matrix is fixed at the start of the algorithm for each block of coordinates.
  \item \textbf{Inexact updates.} To ensure that this method is computationally practical, it is imperative that the iterates are inexpensive, and RCD achieves this through the use of \emph{inexact} updates. Any algorithm can be used to approximately minimize the quadratic model. Moreover, the stopping conditions for inner solve are \emph{easy to verify}; they depend upon quantities that are easy/inexpensive to obtain, or may be available as a byproduct of the inner search direction solver.
  %This is a significant advantage over the stopping condition given in \cite{Tappenden13}, which can be difficult to verify, and restricts the types of algorithms that can be used to solve the subproblem.`
  \item \textbf{Blocks can vary throughout iterations.} If $\Psi(x)$ is completely separable into coordinates then we do not restrict ourselves to a fixed block structure; rather we allow the blocks of coordinates to \emph{change at any iteration}. This is important because every element of the Hessian can be accessed (this is discussed further in Section~\ref{S_HessianApprox}).
  \item \textbf{Line search.} The algorithm includes a line search step to ensure a monotonic decrease of the objective function as iterates progress.
  The line search is inexpensive to perform because, at each iteration, \textit{it depends on a single block of coordinates only}.
  One of the major advantages of incorporating second-order information combined with line search is to allow in practice the selection of \emph{large step sizes} (close to one).
  This is because unit step sizes can substantially improve the practical efficiency of a method. We prove that if $f$ is strongly convex, then close to the optimal solution unit step sizes are selected.  In fact, for all experiments that we performed,
  unit step sizes were accepted by line search for the majority of the iterations.

  \item \textbf{Convergence theory.} We provide global convergence results to show that the RCD algorithm is guaranteed to converge in the limit. We also provide local convergence theory for strongly convex functions $f$. In particular, depending on the choice of stopping condition for the inner search direction solve and the matrix $H\ii(x_k)$, we show that close to the optimal solution RCD has on expectation \textit{block} quadratic or superlinear rate of convergence.
\end{enumerate}
\subsection{Format of the paper}
The paper is organised as follows. In Section \ref{Section_Preliminaries} we introduce the notation and definitions that are used throughout this paper, as well as giving several technical results. We also define the quadratic model that is used in the algorithm, prove the equivalence of some stationarity conditions for problem \eqref{Def_F}, and define a continuous measure of the distance of the current point from the set of solutions of \eqref{Def_F}. A thorough description of the RCD algorithm is presented in Section \ref{S_Algorithm}, including how the blocks are selected/sampled at each iteration, a description of the search direction and line search, several suggestions for the matrices $H\ii(x_k)$, and we also present several concrete examples.

The second half of the paper is devoted to providing convergence results and numerical experiments. In Sections \ref{S_GlobalConvergence}, global convergence results are presented, which do not require $f$ to be convex.
%\textcolor{red}{This section also shows that RCD achieves a linear rate of convergence.}
Local convergence theory for RCD is presented in Section \ref{S_LocalConvergence}. There we show that, close to optimality line search accepts unit step sizes. Moreover, if both the stopping conditions for the inner search direction solve and the matrix $H\ii(x_k)$ are chosen appropriately, then RCD has on expectation block quadratic or superlinear rate of convergence. Finally, several numerical experiments are presented in Section \ref{S_Numerical}, which show that the algorithm performs very well in practice.

\section{Preliminaries}
\label{Section_Preliminaries}
In this section we introduce the notation and definitions that are used in this paper, and we also present some important technical results. Throughout the paper $\|\cdot\| \equiv \sqrt{\langle \cdot, \cdot \rangle}$
and $\|\cdot\|_A \equiv \sqrt{\langle \cdot, A\cdot \rangle}$, where $A$ is a positive definite matrix. Moreover, $\lambda_{\min}(\cdot)$ and $\lambda_{\max}(\cdot)$ denote the smallest  and largest eigenvalue of $\cdot$, respectively.

\subsection{Subgradient and subdifferential}
For a function $\Phi: \R^N \to \R \cup \{+ \infty\}$ the elements $s\in\mathbb{R}^N$ that satisfy
$$
     \Phi(y) \geq \Phi(x) + \langle s,y-x \rangle,
$$
are called the subgradients of $\Phi$ at point $x$. In words, all elements defining a linear function that supports the function $\Phi$ at point $x$ are subgradients. The set of all $s$ at a point $x$ is called the subdifferential of $\Phi$ and it is denoted by $\partial \Phi(x)$.

\subsection{Convexity}
\label{S_StrongConvex}
A function $\Phi: \R^N \to \R \cup \{+ \infty\}$ is strongly convex with convexity parameter $\mu_{\Phi} > 0$ if for all $x,y \in \R^N$, and where $s\in\partial \Phi(x)$,
\begin{equation*}
\label{strongly_convex_1}
     \Phi(y) \geq \Phi(x) + \langle s,y-x \rangle + \tfrac{\mu_{\Phi}}{2}\|y-x\|^2.
\end{equation*}
If $\mu_\Phi = 0$. then function $\Phi$ is said to be convex.

\subsection{Convex conjugate and proximal mapping}
For a convex function $\Phi: \R^N \to \R \cup \{+ \infty\}$, its convex conjugate is defined as
$ \Phi^*(y) \equiv\eqdef \sup_{u\in\mathbb{R}^N} \langle u,y \rangle  - \Phi(u) . $
%If $u^*$ is the supremum of the previous problem then we have that
%\begin{equation}
%\label{eq:2}
%\Phi^*(y) = \left\{
%  \begin{array}{l l}
%    \langle u^*,y \rangle  - \Phi(u^*) & \quad \text{if $y \in \partial \Phi(u^*)$}\\
%    +\infty & \quad \text{otherwise.}
%  \end{array} \right.
%\end{equation}
%In words, the domain of $\Phi^*$ is the subdifferential of $\Phi$
%at a particular point, which depends on input $y$ of $\Phi^*$.
The proximal mapping of a convex function $\Psi$ at $x$ is
\begin{equation}
\label{Def_prox}
  \prox_\Psi (x) \eqdef \arg \min_{y\in \R^N} \Psi(y) + \frac12 \|y-x\|^2,
\end{equation}
and the proximal mapping of its convex conjugate $\Psi^*$ is
\begin{equation}
\label{Def_dual_prox}
  \prox_{\Psi^*} (x) \eqdef \arg \min_{y\in \R^N} \Psi^*(y) + \frac12 \|y-x\|^2.
\end{equation}
The following relation holds between the two proximal mappings.
\begin{lemma}[Chapter 1, ($1.4$) in \cite{Rockafellar06}]
\label{moreau}
  Let $\Psi$ be a convex function and let $\Psi^*$ denote its convex conjugate. Then,
  $
  x = \prox_{\Psi}(x) + \prox_{\Psi^*}(x)
  $
  for all $x$.
 \end{lemma}

   From Chapter $1$ of \cite{Rockafellar06}, we also see that $\prox_{\Psi}(\cdot)$ and $\prox_{\Psi^*}(\cdot)$ are nonexpansive
  \begin{equation}
  \label{Eq_proxnonexpansive}
  \|\prox_{\Psi}(y) - \prox_{\Psi}(x)\| \le \|y -x\|, \quad  \mbox{and} \quad \|\prox_{\Psi^*}(y) - \prox_{\Psi^*}(x)\| \le \|y -x\|.
  \end{equation}
  Finally,  from Chapter $1$ of \cite{Rockafellar06} we have
  \begin{equation}
  \label{eq:14}
  \prox_{\Psi^*}(x) \in \partial \Psi (\prox_{\Psi}(x)).
  \end{equation}

\subsection{Block decomposition of $\R^N$}
\label{S_Block_structure}

Let $U \in \mathbb{R}^{N \times  N}$ be a column permutation of the $N \times N$ identity matrix and further let $U = [U_1,U_2,\dots,U_n]$ be a decomposition of $U$ into $n$ submatrices, where $U_i$ is $N \times N_i$ and $\sum_{i=1}^n N_i = N$. It is clear that any vector $x \in \mathbb{R}^N$ can be written uniquely as
$x = \sum_{i=1}^n U_ix^{(i)},$ where $x^{(i)} \in \R^{N_i}$ and block $i$ denotes a subset of $\{ 1,2,\dots,N\}$. Moreover, these vectors are given by
\begin{equation}\label{U_i}x^{(i)} \eqdef U_i^Tx.\end{equation}

\subsection{Block decomposition of $\Psi$}\label{S_Psi}
The function $\Psi:\R^N \to \R\cup \{+\infty\}$ is assumed to be block separable. That is, we assume that $\Psi(x)$ can be decomposed as:
\begin{equation}
     \label{S2_separable_psi}
\Psi(x) = \sum_{i=1}^n \Psi_i (x^{(i)}),
\end{equation}
where the functions $\Psi_i: \R^{N_i} \to \R\cup \{+\infty\}$ are convex.

Notice that if $n = N$, $\Psi(x)$ is said to be separable (into coordinates), whereas if $n < N$, then $\Psi(x)$ is said to be \emph{block} separable (separable into blocks of coordinates).
%\begin{assumption}
%  If $\Psi(x)$ is block separable (i.e., $n<N$) then we assume that the block decomposition of $\R^N$ defined in Section \ref{S_Block_structure} is fixed throughout the algorithm. However, if $\Psi(x)$ is separable, the block structure can change at each iteration.
%\end{assumption}

The following relationship will be used repeatedly in this work:
   \begin{eqnarray}\label{Eq_PsivsPsii}
   \notag
  \Psi(x+U_it\ii) - \Psi(x) &=& \Big(\sum_{j\neq i}\Psi_j (x^{(j)})+ \Psi_i(x\ii+ t\ii)\Big)-\Big(\sum_{j\neq i}\Psi_j (x^{(j)}) + \Psi_i(x\ii)\Big)\\
   &=& \Psi_i(x\ii+t\ii)-\Psi_i(x\ii).
\end{eqnarray}

\subsection{Block Lipschitz continuity of $f$}

Throughout the paper we assume that the gradient of $f$ is block  Lipschitz, uniformly in $x$. This means that, for all $x \in \R^N$, $i\subseteq \{ 1,2,\dots,n\}$ and $t\ii \in \R^{N_i}$ we have
\begin{equation}
\label{S2_Lipschitz}
     \| \nabla_{i} f(x + U_it\ii) - \nabla_{i} f(x) \| \leq \Lip_{i} \|t\ii\|,
\end{equation}
where $ \nabla_{i} f(x)  \overset{\eqref{U_i}}{=} U_i^T\nabla f(x)$. An important consequence of \eqref{S2_Lipschitz} is the following standard inequality \cite[p.57]{Nesterov04}:
\begin{equation}
\label{S2_upperbound}
     f(x+ U_it\ii) \leq f(x) + \langle \nabla_{i} f(x), t\ii \rangle+ \tfrac{\Lip_{i}}{2}\|t\ii\|^2.
\end{equation}

\subsection{Piecewise Quadratic Model}

For fixed $x \in\R^N$, we define a piecewise quadratic approximation of $F$ around the point $(x + t)\,\in \R^N$ as follows:
\begin{equation}
\label{Def_Q}
  F(x + t)\approx Q(x;t) \eqdef  f(x) + \sum_{i=1}^n Q_i(x,t\ii),
\end{equation}
where
\begin{equation}\label{Def_Qi}
  Q_i(x,t\ii) \eqdef \langle \nabla_i f(x), t\ii \rangle + \frac12 \|t\ii\|_{H\ii(x)}^2 + \Psi_i(x\ii + t\ii),
\end{equation}
and $H\ii(x) \in \R^{N_i \times N_i}$ is \emph{any} positive definite matrix, which possibly depends on $x$. Notice that $Q(x;0) = F(x)$ and that $Q_i(x,t\ii)$ is the quadratic model for block $i$.

\subsection{Stationarity conditions}
The following theorem gives the equivalence of some stationarity conditions of problem \eqref{Def_F}.
\begin{theorem}
\label{thm:optimality_conditions}
The following are equivalent first order optimality conditions of problem \eqref{Def_F}.
\begin{enumerate}
\item[(i)] $\nabla f(x) + s=0$ and $s\in \partial \Psi(x)$,
\item[(ii)] $-\nabla f(x)\in \partial \Psi(x)$,
\item[(iii)] $\nabla f(x) + \frac{1}{\beta}\prox_{(\beta\Psi)^*}\left(x - \beta \nabla f(x) \right) = 0$,
\item[(iv)] $x = \prox_{\beta\Psi}\left(x - \beta\nabla f(x) \right)$,
\end{enumerate}
where $\beta$ is any positive constant.
\end{theorem}
\begin{proof}
It is easy to see that $(i)$ are first-order optimality conditions of problem \eqref{Def_F}, which can be obtained by using the definition of subgradient.
It is trivial to show that $(i) \Longleftrightarrow (ii)$. By Lemma \ref{moreau}, we have that $(iii) \Longleftrightarrow (iv)$. We now show that $(iii) \Longleftrightarrow (ii)$.
We rewrite $(ii)$ as
\begin{equation*}
0\in \beta \nabla f(x) + y - x + \beta \partial \Psi(x) \quad \mbox{and} \quad y=x,
\end{equation*}
which is satisfied if and only if $(iv)$ holds, hence, if and only if $(iii)$ holds.
%========================
%If $(iii)$ holds, then we have that $\prox_{\Psi^*}\big(x - \nabla f(x) \big) = -\nabla f(x)$ and $x = \prox_{\Psi}\big(x - \nabla f(x) \big)$. Using these in \eqref{eq:14} we obtain $(ii)$.
%We now prove the reverse, $(ii) \Longrightarrow (iii)$. We rewrite $(ii)$ as
%\begin{equation}
%\label{eq:101}
%\nabla f(x) + x - y \in \partial \Psi(y), \quad y=x.
%\end{equation}
%The first equations are the optimality conditions of
%$$
%\min_{y\in \R^N} \Psi(y) + \frac12 \|y-x + \nabla f(x)\|^2.
%$$
%Based on the definition of proximal mapping \eqref{Def_prox}, we have that $y=\prox_{\Psi}\left(x - \nabla f(x) \right)$.
%It is easy to observe that the previous and the second equations in \eqref{eq:101} imply that $x=\prox_{\Psi}\left(x - \nabla f(x) \right)$,
%which gives conditions $(iv)$. But conditions $(iv)$ are equivalent to $(iii)$, hence, $(ii) \Longrightarrow (iii)$.
\end{proof}

Let us define the continuous function
\begin{equation}\label{defeq:1}
g(x;t) \eqdef \nabla f(x) + H(x) t + \frac{1}{\beta}\prox_{(\beta\Psi)^*}\big(x + t - \beta (\nabla f(x) + H(x) t) \big),
\end{equation}
where $\beta$ is a positive constant, which is used in the local convergence analysis (Section \ref{S_LocalConvergence}). By Theorem \ref{thm:optimality_conditions}, the points that satisfy $g(x;0) = \nabla f(x) + \frac1\beta \prox_{(\beta\Psi)^*}\left(x - \beta \nabla f(x)\right)=0$ are stationary points for problem \eqref{Def_F}. Hence, $g(x;0)$ is a continuous measure of the distance from the set of stationary points of problem \eqref{Def_F}.

Furthermore, let us define
\begin{eqnarray}
\label{defeq:2}
g_i(x;t\ii) &\eqdef& \nabla_i f(x) + H\ii (x)t\ii  \\
               &          &+ \frac{1}{\beta}\prox_{(\beta\Psi)^*_i}\big(x\ii + t\ii - \beta(\nabla_i f(x) + H\ii(x)t\ii) \big),\notag
\end{eqnarray}
which will be used as a continuous measure for the distance from stationarity of the block piecewise quadratic function $Q_i(x_k;t\ii)$.

\section{The Algorithm}
\label{S_Algorithm}
In this section we  present the Robust Coordinate Descent (RCD) algorithm for solving problems of the form \eqref{Def_F}. There are three key steps in the algorithm: (step $4$) the coordinates are sampled randomly; (step $5$) the quadratic model \eqref{Def_Qi} is solved approximately until the stopping conditions \eqref{Def_stoppingconditions} are satisfied to give a search direction; (step $6$) a line search is performed to find a step size that ensures a sufficient reduction in the objective value. Once these key steps have been performed, the current point $x_k$ is updated to give a new point $x_{k+1}$, and the process is repeated.

The following assumption is used in RCD. The reason this assumption is used will be made clear in Section \ref{S_blockstructure}.
\begin{assumption}
\label{Assump_sampling}
  The block decomposition of $\R^N$ used within RCD, and the associated probability distribution, adhere to the block structure of $\Psi(x)$.
\end{assumption}

We now present pseudocode for the algorithm, while a thorough description of each of the key steps in the algorithm will follow in the rest of this section.

\begin{algorithm}[H]
\begin{algorithmic}[1]
\vspace{0.1cm}
\STATE \textbf{Input} Choose $x_0\in \R^N$, $\theta \in (0,1/2)$ and $\beta>0$. %Notice that $\beta$ is used in definition of $g_i(x;t\ii)$ in \eqref{defeq:2}. \vspace{2mm}
\STATE{\textbf{Initialize} a decomposition of $\R^N$ and a probability distribution following Assumption~\ref{Assump_sampling}}\vspace{2mm}
%\IF{$\Psi(x)$ is block separable with $n < N$}
%\STATE{Fix the block decomposition $U = [U_1,\dots,U_n]$ and set $p_i = \frac{1}{n}$ for all $i = 1,\dots,n$.}
%\ELSE
%\STATE{Set $1\le \tau \le N$ and $p_i = 1/^NC_{\tau}$ for $i = 1,\dots,^NC_{\tau}$.}
%\ENDIF\vspace{2mm}
\FOR{$k=1,2,\cdots$} \vspace{2mm}
\STATE Sample a block of coordinates $i$ with probability $p_i >0$.\\ \vspace{2mm}
\STATE If $g_i(x_k;0)=0$ then go to Step $3$; else approximately solve \begin{equation}\label{eq_subproblem}
t_k\ii \eqdef \argmin_{t\ii} Q_i(x_k;t\ii),
\end{equation}
until the stopping conditions
\begin{equation}\label{Def_stoppingconditions}
Q(x_k;U_it\ii_k) < Q(x_k;0) \quad \mbox{and} \quad \|g_i(x_k;t\ii_k)\| \le \eta_k^i \|g_i(x_k;0)\|,
\end{equation}
are satisfied, (where $\eta_k^i \in [0,1)$).
\STATE Perform a backtracking line search along the direction $t_k\ii$ starting from $\alpha=1$. That is, find $\alpha\in(0,1]$ such that
\begin{equation}\label{Def_linesearch}
F(x_k) - F(x_k+\alpha U_it_k\ii) \ge \theta \left(\ell(x_k;0) - \ell(x_k;\alpha U_it_k\ii)\right),
\end{equation}
where
\begin{equation}\label{Def_lossfunctionli}
\ell(x_k; t) \eqdef f(x_k) + \langle \nabla f(x_k), t \rangle + \Psi(x_k + t).
\end{equation}
\STATE Update $x_{k+1} = x_k + \alpha U_i t_k\ii$
\ENDFOR
\end{algorithmic}
\caption{Robust Coordinate Descent (RCD)}
\label{RCD}
\end{algorithm}
%In words, the algorithm randomly selects a block of coordinates $i$ to update at iteration $k$. An inexact update is found by approximately solving \eqref{eq_subproblem}, and accepting the update $t$ that satisfies the stopping conditions \eqref{Def_stoppingconditions}. A line search is then performed to ensure a sufficient reduction in the objective value and the scaled update is applied to the current point $x_k$ to give a new point $x_{k+1}$.
%\textcolor{red}{Explain step three.}

%The rest of this section is dedicated to providing a description of the key steps of RCD in detail.

\subsection{Block structure and selection of coordinates (Steps $\boldsymbol 2$ \& $\boldsymbol 4$)}
\label{S_blockstructure}
One of the crucial ideas of this algorithm is that the block of coordinates to be updated at each iteration is chosen \emph{randomly}. This allows the coordinates to be selected very quickly. In this section we explain in detail, how the blocks are selected/sampled at each iteration. We also give examples of how coordinates can be randomly sampled such that Assumption \ref{Assump_sampling} is satisfied.

\subsubsection{$\Psi$ is block separable with $n < N$}

When $\Psi$ has a fixed block structure (i.e., $n < N$), the block decomposition of $\R^N$ (via the matrix $U = [U_1,\dots,U_n]$) described in Section \ref{S_Block_structure} is fixed at the start of the algorithm to coincide with the block structure of $\Psi$, and does not change as iterations progress. There are several ways to initialize a sampling scheme to use in RCD that follow Assumption \ref{Assump_sampling}.

\begin{enumerate}
  \item Fix the $n$ blocks of coordinates according to the decomposition of $\R^N$ defined by $U$, and in the algorithm, select each block of coordinates with some probability $p_i$. (e.g., uniform probabilities $p_i = 1/n >0$ for all $i = 1,\dots,n$).
  \item Perform (single) coordinate descent, where at each iteration of RCD, the coordinate $i$ is selected with some probability $p_i$ (e.g., uniform probabilities $p_i = 1/N$ for all $i$).
  \item Perform block coordinate descent, where each block of coordinates has cardinality $N_{\min} \eqdef \min\{N_1,\dots,N_n\}$. The restriction is that, at any iteration $k$, the sampled coordinates forming block $i$, must all belong to the same block of $N_j$ coordinates defined by submatrix $U_j$. (i.e., Assumption \ref{Assump_sampling} is satisfied because the decomposition of $U$ is obeyed.) Recall that $\Psi$ is separable into $n$ blocks. Let the total number of subblocks be $l(>n)$, where we assume that each coordinate $1,\dots,N$ appears in at least one of the $l$ blocks. Then each subblock is selected with probability $p_i$.
\end{enumerate}

\subsubsection{$\Psi$ is separable with $n=N$}

When $\Psi$ is separable into coordinates, we have complete control over the indices that are updated at each iteration.

Let $\tau$ denote the block size (number of coordinates that are updated at any iteration $k$), where $1\leq \tau \leq N$. Note that there are $^NC_{\tau}$ subsets\footnote{Here $^NC_{\tau}$ denotes the usual `N choose $\tau$'. i.e., $^NC_{\tau} = N!/(\tau!(N-\tau))!$} of $\tau$ coordinates that can be made from the set $\{1,\dots,N\}$. At any iteration $k$ of RCD, a subset of coordinates $i_k$ with $|i_k| = \tau$ is sampled with some probability $p_i$ (e.g., uniform probabilities $p_{i} = 1/^NC_{\tau} >0$ for all $i$).
Note that in practice, one never explicitly forms the $^NC_{\tau}$ different blocks in order to randomly pick one with some probability $p_i$. Instead, $\tau$ coordinates are sampled randomly without replacement.

\subsection{The search direction and Hessian approximation (Step $\boldsymbol 5$)}

In this section we describe how RCD determines the search direction. In particular, RCD forms a quadratic model for block $i$, and minimizes the model approximately until the stopping conditions \eqref{Def_stoppingconditions} are satisfied, giving an `inexact' search direction.

We also describe the importance of the choice of matrix $H$, which is an approximate second order information term. From now on, we will often use the shorthand $H_k\ii \equiv H\ii(x_k)$.

\subsubsection{The search direction}

At each iteration the update/search direction is found as follows. The subproblem \eqref{eq_subproblem}, (where $Q_i(x_k;t\ii)$ is defined in \eqref{Def_Qi}) is approximately solved, and the search direction $t_k\ii$ is accepted when the stopping conditions \eqref{Def_stoppingconditions} are satisfied, for some $\eta_k^i \in [0,1)$. Notice that
\begin{eqnarray}\label{eq:1}
  Q(x;U_i t\ii) - Q(x;0) &\overset{\eqref{Def_Q}}{=}& \langle \nabla_i f(x), t\ii \rangle + \frac12 \|t\ii\|_{H\ii}^2 + \Psi(x + U_it\ii) - \Psi(x) \nonumber \\
  &\overset{\eqref{Eq_PsivsPsii}}{=}& \langle \nabla_i f(x), t\ii \rangle + \frac12 \|t\ii\|_{H\ii}^2 + \Psi_i(x\ii + t\ii) - \Psi_i(x\ii).
\end{eqnarray}
Hence, from \eqref{eq:1}, the stopping conditions \eqref{Def_stoppingconditions} depend on block $i$ only, and are therefore inexpensive to verify, meaning that they are \emph{implementable}.
%\textcolor{red}{Transfer the remark below in the algorithm section.}
\begin{remark}\label{Remark_tineq0}\quad
\begin{itemize}
  \item[(i)] At some iteration $k$, it is possible that $g_i(x_k;0)=0$. In this case, it is easy to verify that the optimal solution of subproblem \eqref{eq_subproblem} is $t\ii_k=0$. Therefore, before calculating $t\ii_k$ we check a-priori if condition $g_i(x_k;0)=0$ is satisfied.
  \item[(ii)] Notice that, unless at optimality (i.e., $g(x_k;0)=0$), there will always be blocks $i$ such that $g_i(x_k;0)\neq 0$, which implies that $t\ii_k \neq 0$. Hence, RCD will not stagnate.
  \item[(iii)] Following similar arguments as those made in \cite[p.4]{sqa}, we prove this in Lemma \ref{lem:Fdec} that both conditions are required to ensure that $t_k\ii$ is a descent direction.
  %In particular, condition $\|g_i(x_k;t\ii_k)\| \le \eta_k^i \|g_i(x_k;0)\|$ alone does not guaranteed that $t_k\ii$ is a descent direction. However, imposing the additional condition that quadratic model is decreased at $x_k+t_k\ii$ does ensure that $t_k\ii$ is a descent direction.
\end{itemize}
\end{remark}

\subsubsection{The Hessian approximation}
\label{S_HessianApprox}
Arguably, them most important feature of this method is that the quadratic model \eqref{Def_Qi} incorporates second order information in the form of a positive definite matrix $H_k\ii$. This is key because, depending upon the choice of $H_k\ii$, it makes the method robust. Moreover, at each iteration, the user has complete freedom over the choice of $H_k\ii\succ0$.

We now provide a few suggestions for the choice of $H_k\ii$. (This list is not intended to be exhaustive.) Notice that in each case there is a trade off between a matrix that is inexpensive to work with, and one that is a more accurate representation of the true block Hessian.
\begin{enumerate}
\item Clearly, the simplest option is to set $H_k\ii = I$ for all $i$ and $k$. In this case \emph{no second order information is employed by the method.}
  \item A second option is to let $H_k\ii = \text{diag}(\nabla_i^2 f(x_k))$. In this case $H_k\ii$ and it's inverse are inexpensive to work with. Moreover, if $f$ is quadratic, then $\nabla^2 f(x_k)$ is constant for all $k$, so $H = \text{diag}(\nabla^2 f(x))$ can be computed and stored at the start of the algorithm and elements can be accessed throughout the algorithm as necessary. This is very effective if $\text{diag}(\nabla^2 f(x))$ is a good approximation to $\nabla^2 f(x)$.
  \item A third option is to let $H_k\ii = \nabla_i^2 f(x_k)$ (i.e., $H_k\ii$ is a principal minor of the Hessian). In this case, $H_k\ii$ provides the most accurate second order information, but it is (potentially) more computationally expensive to work with.
  %particularly if $\nabla_i^2 f(x_k)$ is dense.
  %Note that there are no storage requirements or preprocessing required for this approach because $H_k\ii$ is formed/used and then discarded at each iteration.
  In practice the matrix $\nabla_i^2 f(x_k)$ is used in a matrix-free way and is not explicitly stored. For example, there may be an analytic formula for performing matrix-vector products with $\nabla_i^2 f(x_k)$, or techniques from automatic differentiation could be employed, see \cite[Section $7$]{IEEEhowto:wrightbook2}.

  \item Another option is to use a quasi-Newton type approach where $H_k\ii$ is an approximation to $\nabla_i^2 f(x_k)$ based on the limited-memory BFGS update scheme, see \cite[Section $8$]{IEEEhowto:wrightbook2}.
  This approach might be more suitable in cases that the problem is not very ill-conditioned and additionally performing matrix-vector products with $\nabla_i^2 f(x_k)$ is expensive.
\end{enumerate}

\begin{remark}\quad
  If any of the matrices above are not positive definite, then they can be altered to make them so. For example, if $H_k\ii$ is diagonal, any zero that appears on the diagonal can be replaced with a positive number. Moreover, if $\nabla_i^2 f(x_k)$ is not positive definite, a multiple of the identity can be added to it.
\end{remark}

An advantage of the RCD algorithm (if Option 3 is used for $H_k\ii$) is that \emph{all elements of the Hessian can be accessed.} This is because the blocks of coordinates can change at every iteration, and so too can matrix $H_k\ii$. This makes RCD extremely \emph{flexible} and is particularly advantageous when there are large off diagonal elements in the Hessian.

\subsection{The line search (Step $\boldsymbol 6$)}
\label{subsec:pract_line}

The stopping conditions \eqref{Def_stoppingconditions} ensure that $t_k\ii$ is a descent direction, but if the full step $x_k + U_it_k\ii$ is taken, a reduction in the function value \eqref{Def_F}  is not guaranteed. To this end, we include a line search step in our algorithm in order to guarantee monotonic decrease of function $F$. Essentially, the line search guarantees the sufficient decrease of $F$ at every iteration, where sufficient decrease is measured by the loss function \eqref{Def_lossfunctionli}.

In particular, for fixed $\theta \in (0,1/2)$, we require that for some $\alpha \in (0,1]$, \eqref{Def_linesearch} is satisfied. (In Lemma \ref{lem:Fdec} we prove that there exists a subinterval $(0,\tilde{\alpha}]$ of $(0,1]$ in which \eqref{Def_linesearch} is satisfied.) Notice that
\begin{eqnarray}
\notag
  \ell(x;U_it\ii)-\ell(x;0) &\overset{\eqref{Def_lossfunctionli}}{=}& \langle \nabla_i f(x),t\ii \rangle + \Psi(x + U_it\ii) - \Psi(x)\\
  \label{Eq_differenceloss}
  &\overset{\eqref{Eq_PsivsPsii}}{=}& \langle \nabla_i f(x),t\ii\rangle  + \Psi_i(x\ii + t\ii) - \Psi_i(x\ii),
\end{eqnarray}
which shows that the calculation of the right hand side of \eqref{Def_linesearch} only depends upon block $i$, so it is inexpensive. Moreover, the line search condition (Step 5) involves the difference between function values $F(x_k) - F(x_k + \alpha U_i t_k\ii)$. Fortunately, while function values can be expensive to compute, the difference in the objective value between iterates need not be (this is discussed in more detail in Section \ref{S_Examples}).

\subsection{Examples}
\label{S_Examples}
In this section we provide several examples to demonstrate the practicality of the algorithm. These examples demonstrate that  the difference of function values $F(x_k) - F(x_k + \alpha U_i t_k\ii)$
required by the line search conditions \eqref{Def_linesearch}, can be easy/inexpensive to implement and verify.

\subsubsection{Quadratic loss plus regularization example}

Suppose that
$f(x)  = \tfrac12\|Ax-b\|^2$ and $\Psi(x) \neq 0,$
where $A\in\mathbb{R}^{m\times N}$, $b\in\mathbb{R}^m$ and $x \in \R^N$.
Then
\begin{eqnarray}
\label{Eq_Fvaldiff}
  F(x_k) - F(x_k+\alpha U_i t_k\ii) &\overset{\eqref{Eq_PsivsPsii}}{=}& f(x_k) + \Psi_i(x\ii)  \\ \notag
       && - f(x_k+\alpha U_i t_k\ii)  - \Psi_i(x_k\ii+\alpha t_k\ii) \\ \notag
  &=& \Psi_i(x\ii)- \alpha \langle \nabla_i f(x), t_k\ii \rangle - \frac{\alpha^2}2\|A_it_k\ii\|_2^2 \\\notag
  && - \Psi_i(x_k\ii+\alpha t_k\ii).
\end{eqnarray}
Notice that calculation of $F(x_k) - F(x_k+\alpha U_i t_k\ii)$ as a function of $\alpha$ only depends on block $i$, hence, it is inexpensive.
Moreover, in some cases some of the quantities in \eqref{Eq_Fvaldiff} are already needed in the computation of the search direction $t$, so regarding the line search step, they essentially come ``for free''.

\subsubsection{Logistic regression example}
Suppose that
$$f(x) \equiv  \sum_{j=1}^m \log(1 + e^{-b_j a_j^T x}) \quad \mbox{and} \quad \Psi(x) \neq 0,$$
where $a_j^T$ is the $j$th row of a matrix $A\in\mathbb{R}^{m\times n}$
and $b_j$ is the $j$th component of vector $b\in\mathbb{R}^m$. As before, we need to evaluate \eqref{Eq_Fvaldiff}. Let us split calculation of $F(x_k) - F(x_k+\alpha U_i t_k\ii) $ in parts. The first part $\Psi_i(x\ii)- \Psi_i(x_k\ii+\alpha t_k\ii) $ is inexpensive, since it depends only on block $i$. The second part $f(x_k) - f(x_k+\alpha U_i t_k\ii)$ is more expensive because
is depends upon the logarithm.
In this case, one can calculate $f(x_0)$ \textit{once} at the beginning of the algorithm and then update $f(x_k+\alpha U_i t_k\ii)$ $\forall k\ge1$ less expensively.
In particular, let us assume that the following terms:
\begin{equation}
\label{inner_prod_log}
e^{-b_j a_j^Tx_0 } \quad \forall j \quad \mbox{and} \quad f(x_0)=\sum_{j=1}^m \log(1 + e^{-b_j a_j^Tx_0}),
\end{equation}
 are calculated once and stored in memory.
Then, at iteration $1$,  the calculation of
$f(x_0 + \alpha U_i t\ii_0) = \sum_{j=1}^m \log(1 + e^{-b_j a_j^Tx_0}e^{-\alpha b_j a_j^T(U_i t\ii_0)})$
is required for different values of $\alpha$ by the backtracking line search algorithm.
The most demanding task in calculating $f(x_0 + \alpha U_i t\ii_0)$ is the calculation of the products $b_j a_j^T(U_i t\ii_0)$ $\forall j$ \textit{once}, which is inexpensive since $\forall j$ this operation
depends only on block $i$. Having $b_j a_j^T(U_i t\ii_0)$ $\forall j$ and \eqref{inner_prod_log} calculation of $f(x_0) - f(x_0 + \alpha U_i t\ii_0)$ for different values of $\alpha$ is inexpensive.
At the end of the process, $f(x_1)$ and $e^{-b_j a_j^T x_1}$ $\forall j $ will be given for free, and the same process can be followed for the calculation of $f(x_1) - f(x_1 + \alpha U_i t\ii_1)$ etc.

\section{Global convergence theory without convexity of $f$}
\label{S_GlobalConvergence}
In this section we provide global convergence theory for the RCD algorithm. Note that we \emph{do not assume that $f$ is convex}.
Throughout this section we denote $H\ii_k \equiv H\ii(x_k)$. The following assumptions are made about $H_k\ii$ and $f$.

\begin{assumption}\label{Assump_HisPD}
  There exist constants $0 < \lambda_i \leq \Lambda_i$, such that the sequence $\{H_k\ii\}_{k\geq 0}$ satisfies
  \begin{equation}\label{Assumption_lambdai}
    0 < \lambda_i \leq \lambda_{\min}(H_k\ii) \quad \text{and} \quad \lambda_{\max}(H_k\ii) \leq \Lambda_i, \quad \text{for all } i \text{ and } k.
  \end{equation}
\end{assumption}

\begin{assumption}\label{Assump_fisLipschitz}
  The function $f$ is smooth, bounded below, and satisfies \eqref{S2_Lipschitz} for all $i$.
\end{assumption}

Assumption \ref{Assump_HisPD} explains that the Hessian approximation $H_k\ii$ must be positive definite for all blocks $i$ at all iterations $k$. Assumption \ref{Assump_fisLipschitz} explains that $f$ must be block Lipschitz for all blocks $i$ and all iterations $k$.

Before proving global convergence of RCD, we present several technical results. The following lemma shows that if $t\ii_k$ is nonzero, then $F$ is decreased.

\begin{lemma}
\label{lem:Fdec}
  Let Assumptions \ref{Assump_HisPD} and \ref{Assump_fisLipschitz} hold. Let $\theta\in(0,1/2)$ and let $x_k$ and $i$ be generated by RCD. Then Step 6 of RCD will accept a step-size $\alpha$ that satisfies
  \begin{equation}\label{Eq_alphamin}
    \alpha \ge \tilde{\alpha} \qquad \text{where} \qquad \tilde{\alpha}\eqdef(1-\theta)\frac{\lambda_i}{2L_i}.
  \end{equation}
   Furthermore,
  \begin{equation}\label{Eq_F_ubont}
    F(x_k) - F(x_k + \alpha U_it\ii_k)  > \theta(1-\theta)\frac{\lambda_i^2}{4L_i}\|t\ii_k\|^2.
  \end{equation}
\end{lemma}
\begin{proof}
The proof closely follows that of \cite[Theorem $3.1$]{sqa}.
From \eqref{Def_stoppingconditions},
  \begin{equation}
  \label{Eq_yo}
    0 > Q(x_k ; U_it_k\ii) - Q(x_k;0) = \ell(x_k;U_it_k\ii) - \ell(x_k;0) + \frac12\|t_k\ii\|_{H_k\ii}^2.
  \end{equation}
  Rearranging gives
  \begin{equation}\label{in:1}
    \ell(x_k;0) - \ell(x_k; U_it_k\ii) > \frac12\|t_k\ii\|_{H_k\ii}^2 \overset{\eqref{Assumption_lambdai}}{\geq} \frac12\lambda_i\|t_k\ii\|^2.
  \end{equation}
  By Assumption \ref{Assump_fisLipschitz}, for $\alpha \in (0,1)$, we have
  \begin{equation*}
    F(x_k + \alpha U_it_k\ii) \leq f(x_k) + \alpha \langle \nabla_i f(x), t_k\ii \rangle + \frac{\Lip_i}{2}\alpha^2 \|t_k\ii\|^2 + \Psi(x_k+\alpha U_it_k\ii).
  \end{equation*}
  Adding $\Psi(x_k)$ to both sides of the above and rearranging gives
  \begin{eqnarray}
  \label{in:3}
  \notag
    F(x_k) - F(x_k + \alpha U_it_k\ii) &\ge& -\alpha \langle \nabla_i f(x_k), t_k\ii \rangle - \frac{\Lip_i}{2}\alpha^2 \|t_k\ii\|^2  \\ \notag
                                                        & & - \Psi(x_k+\alpha U_it_k\ii) + \Psi(x_k) \\
    &\overset{\eqref{Eq_differenceloss}}{=}& \ell(x_k;0) - \ell(x_k ; \alpha U_it_k\ii) - \frac{\Lip_i}{2} \alpha^2 \|t_k\ii\|^2.
  \end{eqnarray}
 By convexity of $\Psi(x)$ we have that
\begin{equation}\label{Eq_ell_convexity}
  \ell(x_k;0) - \ell(x_k;\alpha U_it_k\ii)\ge \alpha(\ell(x_k;0) - \ell(x_k;U_it_k\ii)).
\end{equation}
  Then
  \begin{eqnarray*}
  \label{in:10}
    F(x_k) - F(x_k + \alpha U_it_k\ii) &-& \theta (\ell(x_k;0) - \ell(x_k ; \alpha U_it_k\ii)) \nonumber \\
    &\overset{\eqref{in:3}}{\geq}& (1-\theta) \big(\ell(x_k;0) - \ell(x_k ; \alpha U_it\ii)\big) - \frac{\Lip_i}{2}\alpha^2 \|t_k\ii\|^2  \nonumber \\
    &\overset{\eqref{Eq_ell_convexity}}{\geq}& \alpha(1-\theta) \big(\ell(x_k;0) - \ell_i(x_k ;  U_it_k\ii)\big) - \frac{\Lip_i}{2}\alpha^2 \|t_k\ii\|^2  \nonumber \\
    & \stackrel{\eqref{in:1}}{>}&  \frac12\big(\alpha(1-\theta)\lambda_i \|t_k\ii\|^2 - \Lip_i\alpha^2 \|t_k\ii\|^2\big) \nonumber \\
    & =& \frac{\alpha}{2}\big((1-\theta)\lambda_i - \Lip_i\alpha\big) \|t_k\ii\|^2.
  \end{eqnarray*}
  From the previous observe that if $\alpha$ satisfies $0 \le \alpha \le (1-\theta)\frac{\lambda_i}{L_i}$,
  then $\alpha$ also satisfies the backtracking line search step of RCD. Suppose that any $\alpha$ that is rejected by the line search is halved for the next line search trial. Then, it is guaranteed that the $\alpha$ that is accepted satisfies \eqref{Eq_alphamin}.

   Since the line search condition \eqref{Def_linesearch} is guaranteed to be satisfied for some step size $\alpha$, from \eqref{Def_linesearch}
   and convexity of $\ell$ we obtain
   $F(x_k) - F(x_k + \alpha U_it_k\ii) \ge \theta \alpha (\ell(x_k;0) - \ell(x_k ; U_it_k\ii))$. Using \eqref{in:1} and \eqref{Eq_alphamin} in the last inequality we obtain \eqref{Eq_F_ubont}.
\end{proof}

The following lemma bounds the norm of the direction $t_k\ii$ in terms of $g_i(x_k,0)$.
\begin{lemma}\label{lem:Fdec2}
Let Assumptions \ref{Assump_HisPD} and \ref{Assump_fisLipschitz} hold. For $\beta>0$, and $x_k$ and $i$ generated by RCD, we have
\begin{equation}\label{in:8}
  \|t_k\ii\| \ge \gamma_i \|g_i(x_k;0)\|, \quad \mbox{where} \quad \gamma_i \eqdef \frac{1-\eta_k^i}{\frac{1}{\beta} + 2\Lambda_i},
\end{equation}
where $\eta_k^i$ is defined in \eqref{Def_stoppingconditions}. Moreover,
\begin{equation}\label{Eq_Fubong}
  F(x_k) - F(x_k + {\alpha}U_it_k\ii) > \theta(1-\theta)\frac{\lambda_i^2}{4L_i} \gamma_i^2 \|g_i(x_k;0)\|^2.
\end{equation}
\end{lemma}
\begin{proof}
This proof closely follows that of \cite[Theorem $3.1$]{sqa}.
  Using the reverse triangular inequality and the fact that $\prox_{\Psi^*_i}(\cdot)$ is nonexpansive, we have that
  \begin{eqnarray*}
  (1-\eta_k^i)\|g_i(x_k;0)\| & \stackrel{\eqref{Def_stoppingconditions}}{\le}& \|g_i(x_k;0)\| - \|g_i(x_k;t\ii_k)\| \\
  & \le& \|g_i(x_k;t_k\ii) - g_i(x_k;0) \| \\
  & \stackrel{\eqref{defeq:2}}{=} & \|H_k\ii t_k\ii + \tfrac{1}{\beta}\prox_{(\beta\Psi)^*_i}\big(x_k\ii + t_k\ii - \beta(\nabla_i f(x_k) + H_k\ii t_k\ii)  \big)  \\
  & &- \tfrac{1}{\beta}\prox_{(\beta\Psi)^*_i}\big(x_k\ii - \beta\nabla_i f(x_k) \big)\| \\
  & \overset{\eqref{Eq_proxnonexpansive}}{\le} & \|H_k\ii t_k\ii\| + \tfrac{1}{\beta}\| (I- \beta H_k\ii)t_k\ii \| \\
  %& \le & (\|H_k\ii\| + \frac{1}{\beta}\| I- \beta H_k\ii\|)\|t_k\ii \|\\
  %& \le & \|t_k\ii\| + 2\|H_k\ii\|\|t_k\ii\| -\|(H_k\ii)^{\frac12}\|\|t_k\ii\|\\
  & \le & (\tfrac{1}{\beta} + 2\|H_k\ii\| )\|t_k\ii\| \\
  & \le & (\tfrac{1}{\beta} + 2\Lambda_i)\|t_k\ii\|.
  \end{eqnarray*}
  Rearranging gives \eqref{in:8}, and combining \eqref{Eq_F_ubont} and \eqref{in:8} gives \eqref{Eq_Fubong}.

\end{proof}
We now have all the tools to prove global convergence of RCD.
\begin{theorem}
\label{thm:1}
  Let Assumptions \ref{Assump_HisPD} and \ref{Assump_fisLipschitz} hold. Then the following hold for RCD:
  \begin{equation*}
     \lim_{k \to \infty} t\ii_k = 0 \quad \forall i \quad \mbox{and} \quad \lim_{k \to \infty} g(x_k;0) = 0.
  \end{equation*}
  with probability one.
\end{theorem}
\begin{proof}
   %Taking the expectation of \eqref{Eq_F_ubont} gives
   %$$
   % \mathbf{E}\left[F(x_k) - F(x_k + {\alpha}_iU_it\ii_k)\;|\;x_k\right] > \frac{\theta(1-\theta)}{4}\mathbf{E}\left[\frac{\lambda_i^2}{L_i}  \|t_k\ii\|^2 \;|\;x_k\right].
    %$$
    %From Remark \ref{Remark_tineq0}(ii) it is guaranteed that $t\ii_k \neq 0$ for some blocks $i$ for iterations $1,\dots,k$. 
    By Lemma \ref{lem:Fdec}, for $t\ii_k \neq 0$ we have that $F(x_k) > F(x_{k} + {\alpha}U_it\ii_k) $. Since $F$ is bounded from below and every block $(i)$ has positive probability to be selected, then
    for $k\to \infty$ the following holds with probability one:
    \begin{equation}
    \label{eq:10}
     \lim_{k\to \infty} F(x_k) - F(x_k + {\alpha}U_it\ii_k) = 0.
    \end{equation}
    Using \eqref{Eq_F_ubont} in combination with \eqref{eq:10} we get that for $k\to\infty$ with probability one $  \|t_k\ii\|\to 0$, which proves the first part.
    Using \eqref{in:8} and $\|t_k\ii\|\to 0$ for $k\to\infty$ with probability one we also get that $\|g_i(x_k;0)\|\to0$ with probability one for $k\to\infty$
    Since $g_i(x_k;0)$ is a block of $g(x_k;0)$, i.e., $ g_i(x_k;0)  \overset{\eqref{U_i}}{=} U_i^Tg(x_k;0)$ and since every block $g_i(x_k;0)$ tends to zero as $k\to\infty$ with probability one, then we have that
    $g(x_k;0)\to0$ with probability one.
    %Note that $\lambda_i^2/L_i>0$ is constant for every $i$, so \eqref{eq:10} implies that
    %$t\ii_k \to 0$ for all $i$ as $k\to \infty.$ Combining this with Lemma \ref{lem:Fdec2} gives that $g_i(x_k;0) \to 0$ for all $i$  as $k\to \infty$. Since this holds for an arbitrary block $i$, the result follows.
\end{proof}

\section{Local convergence theory}
\label{S_LocalConvergence}

In this section we present local convergence theory for RCD. First we discuss some common assumptions that are needed.
Throughout the section we set $H_k\ii = \nabla_{i}^2 f(x_k)$ $\forall i$, where $\nabla_{i}^2 f(x)$ denotes the principal minor of the Hessian $\nabla^2 f(x)$ with row (equivalently column) indices in the subset $i$.
\begin{assumption}\label{Assump_stronglyconvex}
  Function $f$ is strongly convex with strong convexity parameter $\mu_f>0$.
\end{assumption}

By continuity of $f$ we have that $\nabla^2 f(x)$ is symmetric, and by strong convexity of $f$ we have that $\mu_f I \preceq \nabla^2 f(x)$.%, where $I$ denotes the (appropriately sized) identity matrix.
%
%Strong convexity of $F$ may come from $f$ or $\Psi$ or both and we will write $\mu_f$ (resp. $\mu_{\Psi}$) for the strong convexity parameter of $f$ (resp. $\Psi$). Following from \eqref{strongly_convex_1}
%\begin{equation}
%     \label{strongly_convex_4}
%\mu_F \geq \mu_f + \mu_\Psi > 0.
%\end{equation}

The next theorem explains that, if the Hessian $H\equiv \nabla^2 f(x)$ is positive definite, then every principal submatrix of it is also positive definite.

\begin{theorem}[Theorem 4.3.15 in \cite{Horn85}]\label{Thm_eigs}
 Let $A \in \R^{N \times N}$ be a Hermitian matrix, let $r$ be an integer with $1 \le r \le N$, and let $A_r$ denote any $r \times r$ principal submatrix of $A$ (obtained by deleting $N-r$ rows and the corresponding columns from $A$). For each integer $k$ such that $1\le k\le r$ we have
$\lambda_k(A) \le \lambda_k(A_r) \le \lambda_{k+N-r}(A),$
where $\lambda_k(\cdot)$ denotes the $k$th eigenvalue of matrix $\cdot$, and the eigenvalues are ordered $\lambda_1 \le \dots, \lambda_N$.
\end{theorem}

\begin{corollary}
  If $f$ is strongly convex with strong convexity parameter $\mu_f>0$, then
  \begin{equation}
\label{eq:11}
\mu_f I \preceq \nabla_i^2 f(x) \qquad \text{for all } i \subseteq \{1,\dots,N\}.
\end{equation}
\end{corollary}

\begin{assumption}\label{Assump_HkiSC}
We assume that the blocks of the Hessian of $f$ are Lipschitz continuous. This means that for all $x \in \R^N$, $i \subseteq \{1,2,\dots,n\}$ and $t\ii \in \R^{N_i}$ we have
\begin{equation}\label{Def_LipschitzHessian}
 \| \nabla_{i}^2 f(x + U_it\ii) - \nabla_{i}^2 f(x) \| \leq M_i \|t\ii\|.
\end{equation}
\end{assumption}

The following theorem is used to show that the backtracking line search accepts units step sizes close to optimality for any block $i$.
Similarly to \cite{sqa}, in order to prove the previous statement we have to impose sufficient decrease of the quadratic model \eqref{eq_subproblem}
at every iteration. This means that the inexactness conditions in \eqref{Def_stoppingconditions} are replaced by
\begin{eqnarray}
\label{Def_stoppingconditions_suff}
 \xi(\ell(x_k;0)-\ell(x_k;U_it\ii_k)) &\leq& Q(x_k;0) - Q(x_k;U_it\ii_k),\\
 \notag \text{and} \;\; \|g_i(x_k;t\ii_k)\| &\leq& \eta_k^i \|g_i(x_k;0)\|,
\end{eqnarray}
where $\xi\in(\theta,1/2)$ and $\eta_k^i \in [0,1)$. 

Notice that, substituting the equality in \eqref{Eq_yo} into \eqref{Def_stoppingconditions_suff}, gives
\begin{equation}
\label{eq:103}
\frac{1}{2}\|t\ii_k\|_{H\ii_k}^2 \le (1-\xi)\big(\ell(x_k;0)-\ell(x_k;U_it\ii_k)\big).
\end{equation}

\begin{theorem}
\label{thm:unitstepsize}
 Let Assumptions \ref{Assump_stronglyconvex} and \ref{Assump_HkiSC} hold. Let $x_k$ and $i$ be generated by RCD.
 Moreover, let subproblem \eqref{eq_subproblem} of RCD be solved inexactly until the inexactness conditions \eqref{Def_stoppingconditions_suff}
 are satisfied.
 If $\|t\ii_k\| \le 3{\lambda_i}(\xi-{\theta})/{M_i}$, where $\theta \in (0,1/2)$, $\xi\in(\theta,1/2)$
 and $\lambda_i \equiv \lambda_{\min}(H_k\ii) \geq \mu_f>0$ $\forall i$ (by Assumption \ref{Assump_stronglyconvex}), then the backtracking line search step in RCD accepts step sizes
 $\alpha = 1$.
\end{theorem}
\begin{proof}The proof closely follows that of \cite[Lemma $3.3$]{Lee13}.
  Using Lipschitz continuity of $H_k\ii $ we have that
  $$
  f(x_k + U_it\ii_k) \le f(x_k) + \langle \nabla_i f(x_k), t\ii_k \rangle + \tfrac12 \|t_k\ii\|_{H_k\ii}^2 + \tfrac{M_i}{6}\|t\ii_k\|^3.
  $$
  Adding $\Psi(x_k+U_i t\ii_k)$ to both sides gives
  \begin{eqnarray}
  \notag
  F(x_k + U_it\ii_k) &\le& f(x_k) + \langle \nabla_i f(x_k), t\ii_k \rangle + \tfrac12 \|t_k\ii\|_{H_k\ii}^2 + \tfrac{M_i}{6}\|t\ii_k\|^3 + \Psi(x_k + U_i t\ii_k)\\
  \notag
    &=& F(x_k) + \langle \nabla_i f(x_k), t\ii_k \rangle + \tfrac12 \|t_k\ii\|_{H_k\ii}^2 + \tfrac{M_i}{6}\|t\ii_k\|^3 \\
    \notag
                            &  & + \Psi(x_k + U_i t\ii_k) -\Psi(x_k)\\
                            \notag
    &=& F(x_k) - (\ell(x_k;0)-\ell(x_k;U_it\ii_k)) + \tfrac12 \|t_k\ii\|_{H_k\ii}^2 + \tfrac{M_i}{6}\|t\ii_k\|^3\\
    \label{eq:104}
    &\overset{\eqref{eq:103}}{\le}& F(x_k) - \xi (\ell(x_k;0)-\ell(x_k;U_it\ii_k)) + \tfrac{M_i}{6}\|t\ii_k\|^3.    
%    &\overset{\eqref{in:1}}{\le}& F(x_k) - (\ell(x_k;0)-\ell(x_k;U_it\ii_k))  + \tfrac12(\ell(x_k;0)-\ell(x_k;U_it\ii_k))\\
%                              &  + &   \tfrac{M_i}{3 \lambda_i}\|t\ii_k\|\left(\ell(x_k;0)-\ell(x_k;U_it\ii_k)\right)\\
%    &=& F(x_k) - \left(\tfrac12 - \tfrac{M_i}{3 \lambda_i}\|t\ii_k\|\right) \left(\ell(x_k;0)-\ell(x_k;U_it\ii_k)\right).
  \end{eqnarray}
  
Clearly, $\lambda_i \equiv \lambda_{\min}(H_k\ii) \leq \|t_k\ii\|_{H_k\ii}^2/\|t_k\ii\|^2$. Using this with \eqref{eq:103} in \eqref{eq:104} we get
  \begin{eqnarray*}
F(x_k + U_i t\ii_k) &\le& F(x_k) - \xi (\ell(x_k;0)-\ell(x_k;U_it\ii_k)) \\
                              & &+ \tfrac{M_i}{3 \lambda_i}\|t\ii_k\|(\ell(x_k;0)-\ell(x_k;U_it\ii_k)) \\
                              &=&F(x_k) - (\xi - \tfrac{M_i}{3 \lambda_i}\|t\ii_k\|)(\ell(x_k;0)-\ell(x_k;U_it\ii_k)).
  \end{eqnarray*}
  If $\|t\ii_k\| \le 3{\lambda_i}(\xi-{\theta})/{M_i}$, $\theta \in (0,\frac12)$ and $\xi\in(\theta,1/2)$ then
  $
  F(x_k) - F(x_k + U_i t\ii_k)  \ge \theta (\ell(x_k;0)-\ell(x_k;U_it\ii_k)),
  $
  which implies that RCD accepts a step $\alpha = 1$.
\end{proof}
\begin{corollary}
\label{corollary:unit_step}
By Theorem \ref{thm:1}, $\|t\ii_k\| \to 0$ $\forall i$ as $k\to \infty$.
Thus, there will be a region close to the optimal solution $x_*$ such that $\|t\ii_{k}\| \le  3 {\lambda_i}(\xi-\theta)/{M_i}$ for all $i$ and for all $k$. Hence, by Theorem \ref{thm:unitstepsize}, in this region, the backtracking line search step in RCD accepts unit step sizes for any $i$.
\end{corollary}

The following assumption is mild since it is guaranteed to be satisfied by Corollary \ref{corollary:unit_step}.

\begin{assumption}\label{Assump_unitstep}
  Iteration $x_k$ is close to the optimal solution $x_*$ of \eqref{Def_F}
such that unit step sizes are accepted by the backtracking line search algorithm of RCD.
\end{assumption}

The next lemma is a technical result that will be used in Theorem \ref{thm:quadratic}.
\begin{lemma}\label{strong_monotonicity}
Let Assumptions \ref{Assump_fisLipschitz} and \ref{Assump_stronglyconvex} hold. Let $L_{\max} = \max_i\{L_1,\dots,L_n\}$. Let $\beta < 1/L_{\max}$ in \eqref{defeq:2}. Then $g_i(x;t\ii)$ inherits strong monotonicity of $\nabla_i f(x)$ $\forall i$:
$$
(u - v)^T (g_i(x;u) - g_i(x;v)) \ge \frac{\mu_f}{2}\|u-v\|^2 \quad \forall u,v \in \R^{N_i}.
$$
\end{lemma}
\begin{proof}
The proof is the same as \cite[Lemma $3.9$]{Lee13}, but restricted to the $i$th block, so is omitted.
\end{proof}

In the following theorem we demonstrate that RCD has on expectation \textit{block} quadratic or superlinear local rate of convergence.
\begin{theorem}
\label{thm:quadratic}
 Let Assumptions \ref{Assump_fisLipschitz}, \ref{Assump_stronglyconvex}, \ref{Assump_HkiSC} and \ref{Assump_unitstep} hold. 
 Let $x_{k+1} = x_k + t_k\ii$, $\eta_k^i = \min\{1/2, \|g_i(x_k;0)\|\},$ and $\beta < 1/L_{\max}$ in \eqref{defeq:2}. Then, $\|g_{i}(x_k;0)\|$ has a quadratic rate of convergence in expectation:
  $$
   \lim_{k\to \infty} \mathbf{E}\left[\frac{\|g_i(x_{k+1};0)\|}{\|g_i(x_{k};0)\|^2}\;|\;x_k\right] = \mbox{c},
  $$
  where $c$ is a positive constant.
  If $\eta_k^i \to 0$ for $k\to \infty$, then $\|g_{i}(x_k;0)\|$ has a superlinear rate of convergence in expectation:
    $$
   \lim_{k\to \infty} \mathbf{E}\left[\frac{\|g_i(x_{k+1};0)\|}{\|g_i(x_{k};0)\|}\;|\;x_k\right] = 0.
  $$
\end{theorem}
\begin{proof}
  For a given $x_k$ we define
  \begin{equation}
    x(\delta) \eqdef x_k + \delta U_i t\ii_k, \quad \text{and} \quad x(\sigma) \eqdef x_k + \sigma U_i t\ii_k.
  \end{equation}
  Using the Fundamental Theorem of Calculus (F.T.o.C.), we have
  \begin{eqnarray}\label{Eq_FToC}
  \notag
    g_i(x(\delta);0) &=& \nabla_i f(x_k) + \delta \nabla_i^2 f(x_k) t\ii_k + \int_0^\delta \int_0^u \nabla_i^3 f(x(\sigma))[t\ii_k,t\ii_k] d \sigma d u \\
                            && + \frac{1}{\beta}\prox_{(\beta\Psi)^*_i}(x\ii(\delta) - \beta\nabla_i f(x(\delta)) ).
    %\\
%  &\phantom{=}& + \prox_{\Psi^*}(\nabla_i f(x(\delta)) - x\ii(\delta))\\\\
%  &= &\nabla_i f(x_k) + \delta^T \nabla_i^2 f(x_k) t\ii_k + \prox_{\Psi^*}(\nabla_i f(x_k) - H_k\ii t\ii_k -x_k\ii - t\ii_k)\\
%  &\phantom{=}& +\prox_{\Psi^*}(\nabla_i f(x(\delta)) - x\ii(\delta)) - \prox_{\Psi^*}(\nabla_i f(x_k) - H_k\ii t\ii_k -x_k\ii - t\ii_k)\\
%  &\phantom{=}& +\int_0^\delta \int_0^u \nabla_i^3 f(x(\sigma)[t\ii_k,t\ii_k] d \delta d u
  \end{eqnarray}
   Also, from the definition of a derivative we have
\begin{eqnarray}
    \label{eqn_intupperbound}
  \notag
    \int_0^\delta \int_0^u \|\nabla_i^3 f(&x(\sigma)&)[t\ii_k,t\ii_k] \| d \sigma d u\\
    \notag
     &=&  \int_0^\delta \int_0^u \lim_{\sigma \to 0}\|\frac{1}{\sigma}(t\ii_k)^T(\nabla_i^2 f(x(\sigma))-\nabla_i^2 f(x_k)) \| d \sigma d u\\
      &\le& \notag \|t\ii_k\|\int_0^\delta \int_0^u \lim_{\sigma \to 0}\|\frac{1}{\sigma}(\nabla_i^2 f(x(\sigma))-\nabla_i^2 f(x_k)) \| d \sigma d u \\ &\overset{\eqref{Def_LipschitzHessian}}{\leq}&
    M_i \|t\ii_k\|^2 \int_0^\delta \int_0^u  1\; d \sigma d u = \frac{\delta^2}{2} M_i \|t\ii_k\|^2.
  \end{eqnarray}
  
  Now, adding and subtracting $\tfrac1\beta \prox_{(\beta\Psi)^*_i}(x_k\ii + t\ii_k - \beta(\nabla_i f(x_k) + H_k\ii t\ii_k) )$ from \eqref{Eq_FToC}, followed by taking norms, applying the triangle inequality, and using \eqref{eqn_intupperbound}, gives
  %\begin{eqnarray}
%       \label{eqn_normFqi}
%     \|g_i(x(\delta);0)\| &\le&
%     \|\nabla_i  f(x_k) + \delta \nabla_i ^2 f(x_k) t\ii_k \\\notag
%     && + \frac{1}{\beta}\prox_{(\beta\Psi)^*_i}(x_k\ii + t\ii_k - \beta(\nabla_i  f(x_k) + H_k\ii t\ii_k) )\|\\
%     \notag
%     && +\frac{1}{\beta} \|\prox_{(\beta\Psi)^*_i}(x\ii(\delta) - \beta\nabla_i  f(x(\delta)) ) \\ \notag
%     && -\prox_{(\beta\Psi)^*_i}(x_k\ii + t\ii_k - \beta(\nabla_i  f(x_k) + H_k\ii t\ii_k) )\|\\ \notag
%     &&+\int_0^\delta \int_0^u \|\nabla_i^3 f(x(\sigma)[t\ii_k,t\ii_k] \| d \sigma d u
%  \end{eqnarray}
%
%
% \begin{eqnarray}
%       \label{eqn_normFqi}\notag
%     &\phantom{\leq}&\|g_i(x(\delta);0)\| \\
%     \notag
%     &\le& \|\nabla_i  f(x_k) + \delta \nabla_i ^2 f(x_k) t\ii_k  + \frac{1}{\beta}\prox_{(\beta\Psi)^*_i}(x_k\ii + t\ii_k - \beta(\nabla_i  f(x_k) + H_k\ii t\ii_k) )\|\\
%     \notag
%     &+& \frac{1}{\beta} \|\prox_{(\beta\Psi)^*_i}(x\ii(\delta) - \beta\nabla_i  f(x(\delta)) )  -\prox_{(\beta\Psi)^*_i}(x_k\ii + t\ii_k - \beta(\nabla_i  f(x_k) + H_k\ii t\ii_k) )\|\\ \notag
%     &+&\int_0^\delta \int_0^u \|\nabla_i^3 f(x(\sigma)[t\ii_k,t\ii_k] \| d \sigma d u
%  \end{eqnarray}
%  Replacing \eqref{eqn_intupperbound} in \eqref{eqn_normFqi} gives
  \begin{eqnarray}\label{Eq_yo2}
  \notag
    \|g_i(x(\delta);0)\|&\le & \|\nabla_i f(x_k) + \delta \nabla_i^2 f(x_k) t\ii_k \\
    \notag
      && +\frac{1}{\beta} \prox_{(\beta\Psi)^*_i}(x_k\ii + t\ii_k - \beta(\nabla_i f(x_k) + H_k\ii t\ii_k) )\|\\
      \notag
     && + \frac{1}{\beta}\|\prox_{(\beta\Psi)^*_i}(x\ii(\delta) - \beta\nabla_i f(x(\delta)) )  \\
     && -\prox_{(\beta\Psi)^*_i}(x_k\ii + t\ii_k - \beta(\nabla_i f(x_k) + H_k\ii t\ii_k) )\| +\frac{\delta^2}{2} M_i \|t\ii_k\|^2.
  \end{eqnarray}
  By Assumption \ref{Assump_unitstep}, RCD accepts unit step sizes. Hence, setting $\delta = 1$ in \eqref{Eq_yo2} gives
  \begin{eqnarray*}
    \|g_i(x_{k+1};0)\| &\le& \|g_i(x_k;t\ii_k)\| + \frac{1}{2}M_i \|t\ii_k\|^2 \\
    && +\frac{1}{\beta}\|\prox_{(\beta\Psi)^*_i}(x_{k+1}\ii - \beta\nabla_i f(x_{k+1}))  \\
    && -\prox_{(\beta\Psi)^*_i}(x_{k+1}\ii - \beta(\nabla_i f(x_k) + H_k\ii t\ii_k) )\|\\
    &\le& \|g_i(x_k;t\ii_k)\| + \|\nabla_i f(x_k) + H_k\ii t\ii_k - \nabla_i f(x_{k+1})\| + \frac12 M_i \|t\ii_k\|^2.
  \end{eqnarray*}
  Using the same trick as before with the F.T.o.C. we have the bound\\
  $\|\nabla_i f(x_k) + H_k\ii t\ii_k - \nabla_i f(x_{k+1})\| \le \frac12 M_i \|t\ii_k\|^2,$
  so that
  \begin{eqnarray}
  \notag
    \|g_i(x_{k+1};0)\| &\le& \|g_i(x_k;t\ii_k)\| + M_i \|t\ii_k\|^2\\
     \label{eq:3}
      &\overset{\eqref{Def_stoppingconditions_suff}}{\le}&  \eta_k^i \|g_i(x_k;0)\| + M_i \|t\ii_k\|^2.
      %
      %&\overset{\text{Lemma } \ref{lem:tech1quad}}{\le}& \eta_k^i \|g_i(x_k;0)\| + \tfrac{8 M_i}{\mu_F}(F(x_k) - F^*).
  \end{eqnarray}
  Setting $u=t\ii_k$ and $v=0$ in Lemma \ref{strong_monotonicity} gives
  $  (g_i(x_k;t\ii_k) - g_i(x_k;0))^Tt\ii_k \ge \frac{\mu_f}{2}\|t\ii_k\|^2$. Then, using Cauchy-Schwarz we have
  $$
  \|g_i(x_k;t\ii_k) - g_i(x_k;0)\| \ge \frac{\mu_f}{2}\|t\ii_k\|.
  $$
  By the triangular inequality and stopping conditions \eqref{Def_stoppingconditions} we have
  \begin{equation}
  \label{eq:102}
  (1+\eta^i_k) \|g_i(x_k;0)\| \ge \frac{\mu_f}{2}\|t\ii_k\|.
  \end{equation}
  Replacing \eqref{eq:102} in \eqref{eq:3} we have
  \begin{equation}
  \label{eq:105}
  \|g_i(x_{k+1};0)\| \le \eta_k^i \|g_i(x_k;0)\| + \frac{4 M_i (1 + \eta^i_k)^2}{\mu_f^2} \|g_i(x_k;0)\|^2.
  \end{equation}
%\textcolor{red}{Fix proof using strong monotonicity of $g$.} Define $P(x_k)$ to be the projection operator onto the set $\partial F(x_k)$.
% Then by setting $s = P(x_k)U_ig_i(x_k;0)$ in the result of Corollary \ref{cor:convex} and replacing the result in \eqref{eq:3}
% we have
%   $$
%    \|g_i(x_{k+1};0)\| \le \eta_k^i \|g_i(x_k;0)\| + \frac{4 M_i}{\mu_F^2}\|P(x_k)U_ig_i(x_k;0)\|^2.
% $$
% Notice that $\|P(x_k)U_ig_i(x_k;0)\|^2 \le \|U_ig_i(x_k;0)\|^2 = \|g_i(x_k;0)\|^2$.
 Moreover, by setting $\eta_k^i = \min\{1/2, \|g_i(x_k;0)\|\}$ we obtain
 \begin{equation}
 \label{eq:4}
    \|g_i(x_{k+1};0)\| \le \left(1 +  \frac{9 M_i}{\mu_f^2} \right)\|g_i(x_k;0)\|^2.
 \end{equation}
 Rearranging and taking expectation gives
 \begin{equation}
 \label{in:15}
 \lim_{k\to \infty} \mathbf{E}\left[\frac{\|g_i(x_{k+1};0)\|}{\|g_i(x_{k};0)\|^2}\;|\;x_k\right] \le \lim_{k\to \infty}  \mathbf{E}\left[1 + \frac{9 M_i}{\mu_F^2}\;|\;x_k\right].
 \end{equation}
 The right hand side of \eqref{in:15} is constant for all $i$, which implies quadratic convergence of $\|g_i(x_k;0)\|$ in expectation.

 Moreover, if $\eta_k^i \to 0$ as $k\to \infty $, then from \eqref{eq:105}, Theorem \ref{thm:1} and Lemma \ref{lem:Fdec2},
  $$
 \lim_{k\to \infty} \mathbf{E}\left[\frac{\|g_i(x_{k+1};0)\|}{\|g_i(x_{k};0)\|}\;|\;x_k\right] = 0,
 $$
 which implies superlinear convergence of $\|g_i(x_k;0)\|$ in expectation.
\end{proof}

% Generally, Theorem \ref{thm:quadratic} does not imply fast \emph{rate} of convergence of $g(x_k;0)$. However,
% in some special cases fast rate of convergence is possible. In particular, let
% the optimal solution $x_*$ of \eqref{Def_F} be sparse and $\Omega$
% be the set of indices where $x_*$ is non zero.  If function $\Psi$
% is separable on $\Omega$, then problem \eqref{Def_F} is equivalent to
% \begin{equation}
%\label{Def_F2}
%  \min_{x^\Omega \in \R^{|\Omega|}}  f(U_\Omega x^\Omega) + \Psi_\Omega(x^\Omega).
%\end{equation}
%For this reduced problem it is only $g_\Omega(x_k;0)$ that measures the distance from stationarity.
%Let us assume now that RCD is being run on problem \eqref{Def_F} with fixed
% size of blocks $i$ which satisfies $s\ge |\Omega|$ and there exists a block $i= \Omega$.
% Then, Theorem \ref{thm:quadratic} implies point-wise convergence of $\|g_i(x_{k+1};0)\|/\|g_i(x_{k};0)\|$ for all blocks (events) $i$, hence for $\Omega$ as well. Hence convergence rate
% of $g_i(x_k;0)$ to zero is quadratic or superlinear depending on the setting of $\eta_k^i$. A well-known example is the case of $\Psi(x) = \|x\|_1$, where the $\ell_1$-norm is fully separable in $N$ blocks and also promotes sparsity on $x_*$.

\section{Numerical Experiments}
\label{S_Numerical}
In this section we examine the performance of two versions of RCD and two versions of a Uniform Coordinate Descent method (UCDC) \cite{petermartin}
on two common optimization problems. The first problem is an $\ell_1$-regularized least squares problem of the form \eqref{Def_F} with
\begin{equation}
\label{least_squares}
f(x) =\tfrac12\|Ax-b\|^2 \quad \mbox{and} \quad \Psi(x) = c \|x\|_1,
\end{equation}
where $c>0$, $x\in\mathbb{R}^N$, $A\in\mathbb{R}^{m\times N }$ and $b\in\mathbb{R}^m$. The second problem is an $\ell_1$-regularized logistic regression problems of the form \eqref{Def_F} with
\begin{equation}
\label{logistic_regression}
f(x) = \sum_{j=1}^m\log(1+e^{-b_jx^T a_j}) \quad \mbox{and} \quad \Psi(x) =c \|x\|_1,
\end{equation}
where $c>0$, $a_j\in\mathbb{R}^N$ $\forall j=1,2,\ldots,m$ are the training samples and $b_j\in\{-1,+1\}$ are the corresponding labels.

For \eqref{least_squares} a synthetic sparse large scale experiment is performed and for \eqref{logistic_regression} we compare the methods on two real world large scale problems from machine learning. Notice that for both \eqref{least_squares} and \eqref{logistic_regression}, $\Psi(x) = c\|x\|_1$, which is fully separable into coordinates. This means that, for RCD, we have complete control over the block decomposition, and the indices making up each block can change at every iteration.

All algorithms are coded in MATLAB, and for fairness, MATLAB is limited to a single computational thread for each test run.
All experiments are performed on a Dell PowerEdge R920 running Redhat Enterprise Linux with four Intel Xeon E7-4830 v2 2.2GHz, 20M Cache, 7.2 GT/s QPI, Turbo (4x10Cores).

\subsection{Implementations of RCD and UCDC}
In this section we discuss some details of the implementations of methods RCD and UCDC.

\subsubsection{RCD}

For the RCD method, we fix the size of blocks $\tau>1$ (to be given in the numerical experiments subsections), and at every iteration of RCD, $\tau$ coordinates are sampled uniformly at random without replacement.

We implement two versions of RCD, which we denote by RCD v.1 and RCD v.2. The two versions only differ in how matrix $H_k\ii$ is chosen. In particular, for RCD v.1 we set $H_k\ii\eqdef \text{diag}(\nabla_i^2 f(x_k))$ for all $i$ and $k$. In this case subproblem \eqref{eq_subproblem} is separable
and it has a closed form solution
$$t\ii_k = \mathcal{S}(x\ii_k - (H_k\ii)^{-1}\nabla_i f(x\ii_k), c\, \text{diag}((H_k\ii)^{-1})),$$
where
\begin{equation}
\label{soft_thresholding}
 \mathcal{S}(u,v) = \mbox{sign}(u) \max(|u| - v,0)
\end{equation}
is the well-known soft-thresholding operator which is applied component wise when $u$ and $v$ are vectors.
Notice that since the subproblem is solved exactly there is no need to verify the stopping conditions \eqref{Def_stoppingconditions}.

For RCD v.2, we set
\begin{equation}
\label{reg_H_k}
H_k\ii \eqdef \nabla_i^2 f(x_k) + \rho I_{N_i}, \text{ for all } i \text{ and } k,
\end{equation}
where $\rho>0$ guarantees that $H_k\ii$ is positive definite for all $i,k$. Hence, the subproblem \eqref{eq_subproblem} is well defined. The larger $\rho$ is the smaller the condition number of matrix $H_k$ becomes, hence, the faster subproblem \eqref{eq_subproblem} will be solved by an iterative solver. However, we do not want $\rho$ to dominate matrix $H_k$ because the essential second order information from $\nabla^2 f(x_k)$ will be lost.
%In all numerical experiments we set $\rho = 10^{-6}$.

In this setting of matrix $H_k\ii$ we solve subproblems \eqref{eq_subproblem} iteratively using an Orthant Wise Limited-memory Quasi-Newton (OWL) method,
which can be downloaded from \url{http://www.di.ens.fr/~mschmidt/Software/L1General.html}.
We chose OWL because it has been shown in \cite{sqa} to result in a robust and efficient deterministic version of RCD, i.e. $\tau=N$ (one block of size $N$). Note that we never explicitly form matrix $H_k\ii$, we only perform matrix-vector products with it in a matrix-free manner.

\subsubsection{UCDC}

We also implement two versions of a uniform coordinate descent method as it is described in Algorithm $2$ in \cite{petermartin}. For both versions the size $\tau$ of the blocks and the decomposition of $\mathbb{R}^N$ into $\ceil{N/\tau}$ blocks are \emph{fixed a-priori} and all blocks are selected by UCDC with uniform probability. We compare two versions of UCDC, denoted by UCDC v.1 and UCDC v.2 respectively. For UCDC v.1 we set $\tau = 1$ and for UCDC v.2 we set $\tau>1$ (the exact $\tau$ is given later).
%For UCDC, the update $t_k\ii $ is computed by solving the subproblem:
%\begin{equation}
%  t_k\ii \eqdef \arg \min_{t\ii \in \R^{N_i}} \{f(x_k) + \langle \nabla_i f(x), t\ii \rangle + \tfrac{L_i}2 \|t\ii\|_{B_i}^2 + \Psi_i(x\ii + t\ii)\},
%\end{equation}
%where $L_i$ is the Lipschitz constant for block $i$ and $B_i$ is a user defined positive definite matrix. Note that, because the block decomposition is fixed for UCDC, $L_i$ and $B_i$ are fixed for each block and do not change as iterations progress.

One of the key ingredients of UCDC are the block Lipschitz constants, which are explicitly required in the algorithm. For single coordinate blocks, the Lipschitz constants can be computed with relative ease. However, for blocks of size greater than 1, the block Lipschitz constants can be far more expensive to compute. (For example, for problem \eqref{least_squares}, the block Lipschitz constants correspond to the maximum eigenvalue of $A_i^TA_i$, where $A_i \eqdef AU_i$.) For this reason, we do not compute the actual block Lipschitz constants,  rather, we use an overapproximation.

To this end, let $L_j>0$ $\forall j=1,2,\cdots,N$ denote the coordinate Lipschitz constants of function $f$. Then the direction $t\ii$ at every iteration is obtained by solving exactly subproblem \eqref{eq_subproblem} with
\begin{equation}
\label{H_k_ucdc}
H_k\ii:= \Big(\sum_{j\in i}L_j\Big) I_\tau,
\end{equation}
using operator \eqref{soft_thresholding}. Notice that for problem \eqref{least_squares}, \eqref{H_k_ucdc} is equivalent to $H_k\ii = \text{trace}(A_i^TA_i) I_{\tau}$, where $\text{trace}(A_i^TA_i)$ is an overapproximation of the maximum eigenvalue of $A_i^TA_i$.

Moreover, notice that Algorithm $2$ in \cite{petermartin} is a special case of RCD where the subproblem \eqref{eq_subproblem} is solved exactly
and line search is unnecessary. This is because, by setting $H_k\ii$ as in \eqref{H_k_ucdc}, then subproblem \eqref{eq_subproblem} is an over estimator of function $F$ along block coordinate direction $t\ii$ (for details we refer the reader to \cite{petermartin}).

\subsection{Termination Criteria and Parameter Tuning}
The only termination criteria that RCD and UCDC should have are maximum number of iterations or maximum running time. This is because using subgradients as a measure of distance from optimality or any other operation of similar cost are considered as expensive tasks for large scale problems. In our experiments RCD and UCDC are terminated when their running time exceeds the maximum allowed running time.
%In our experiments
%RCD and UCDC are terminated when the relative error of the current iteration $x_k$ drops below a certain tolerance.
%However, we do not include the CPU time of calculating the relative error at every iteration in the total CPU time.
Furthermore, for RCD we set parameter $\eta\ii_k$ in \eqref{Def_stoppingconditions} equal to $0.9$ $\forall i,k$ and
$\rho = 10^{-6}$ in \eqref{reg_H_k}.
The maximum number of backtracking line search iterations is set to 10 and $\theta=10^{-3}$.
For UCDC the coordinate Lipschitz constants $L_j$ $\forall j$ are calculated once at the beginning of the algorithm and this task is included in the overall running time.
Finally, all methods are initialized with the zero solution.

\subsection{$\ell_1$-Regularized Least Squares}
In this subsection we present the performance of RCD and UCDC on the $\ell_1$-regularized least squares problem \eqref{least_squares}.
%of the form \eqref{Def_F}, with
%\begin{equation*}
%%\label{least_squares}
%f(x) =\tfrac12\|Ax-b\|^2 \quad \mbox{and} \quad \Psi(x) = c \|x\|_1,
%\end{equation*}
%where $c>0$, $x\in\mathbb{R}^N$, $A\in\mathbb{R}^{m\times N }$ and $b\in\mathbb{R}^m$.
For this problem the data $A$ and $b$ were synthetically constructed using a generator proposed in \cite[Section $6$]{nesterovgen}, and we set $c=1$.
The advantage of this generator is that it produces data $A$ and $b$
with a known minimizer $x_*$. We slightly modified the generator so that we could control the density of $A$.

The dimensions of the problem are $N=2^{21}$ and $m=N/4$ and the generated matrix $A$ is full rank (with at least one non zero component per column) and a density of $\approx10^{-4}mN$. The optimal solution is set to have $\ceil{0.01N}$ non zero components with values uniformly at random in the interval $[-1,1]$. For UCDC, the coordinate Lipschitz constants are
$
L_j := \|A_j\|_2^2 \quad  j=1,2,\cdots,N,
$
and for RCD v.1, RCD v.2 and UCDC v.2, we set $\tau = \ceil{0.01 N}$.

The result of this experiment is shown in Figure \ref{fig1}.
\begin{figure}[h!]%
\centering
\subfloat[Objective function $F(x)$ against iterations]{%
\label{fig1a}%
\includegraphics[scale=0.32]{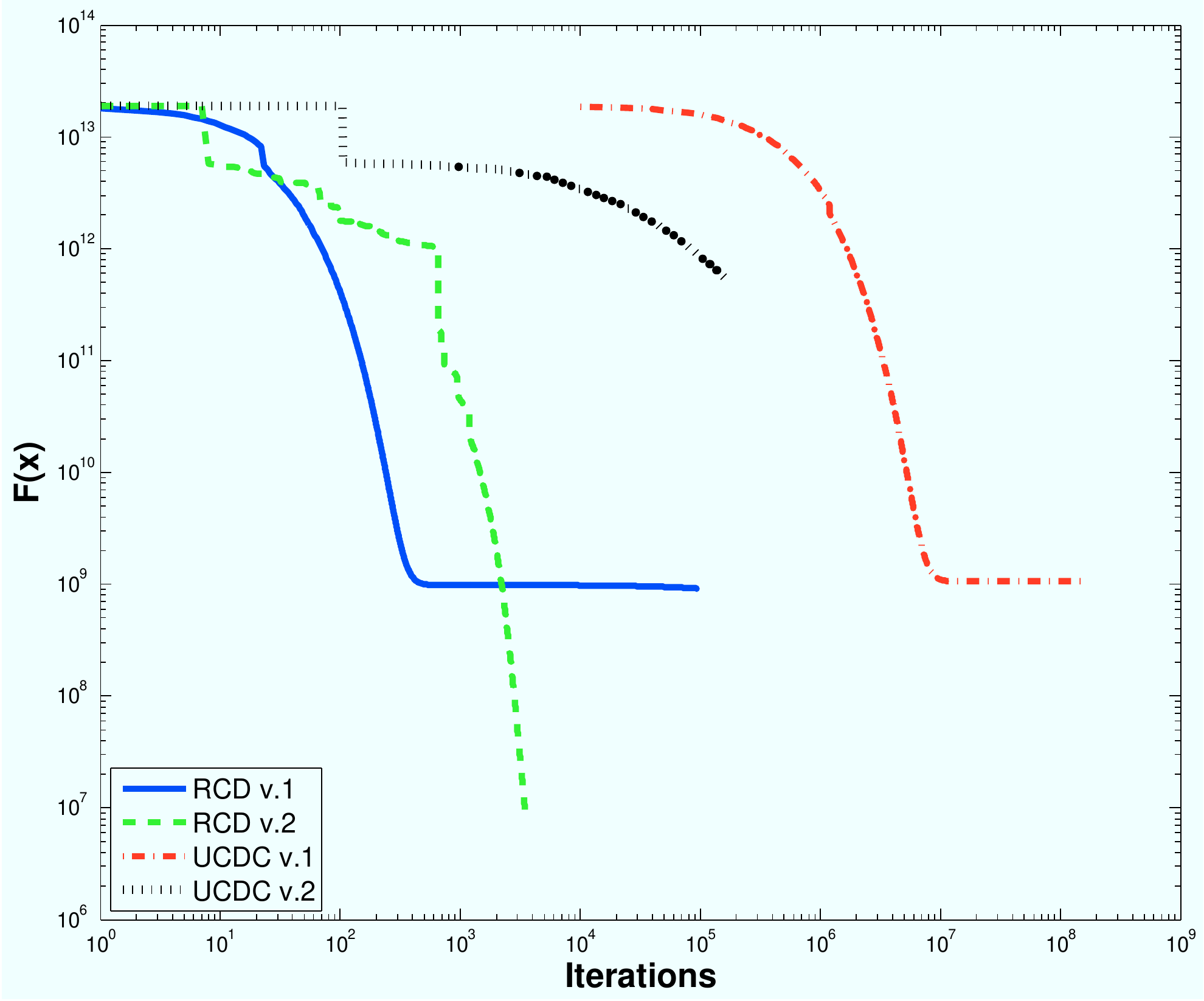}}
\subfloat[Objective function $F(x)$ against time]{%
\label{fig1b}%
\includegraphics[scale=0.32]{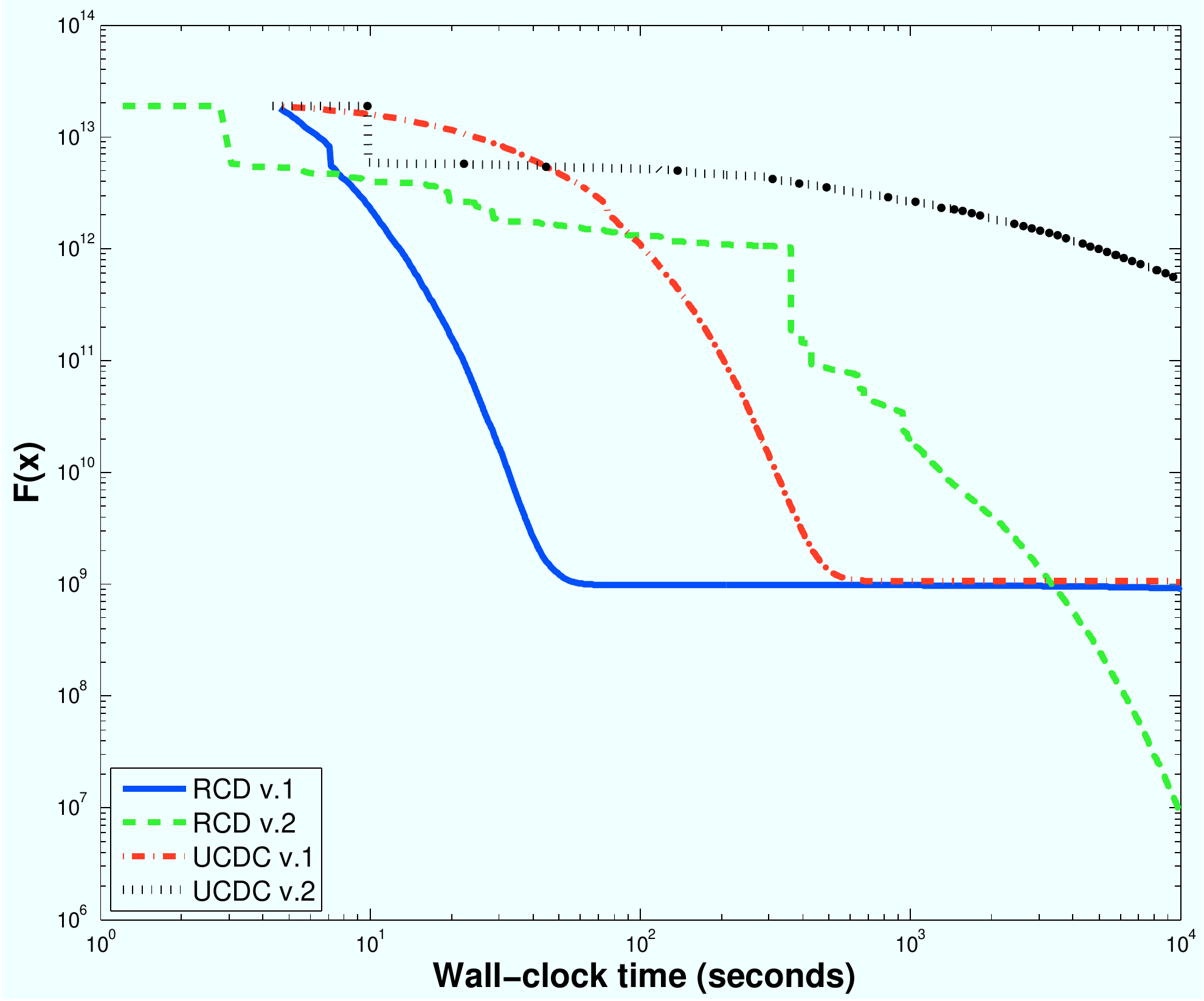}}\\
\caption{Performance of all four methods RCD v.1 and v.2 and UCDC v.1 and v.2 on a sparse large scale $\ell_1$-regularized least squares problem.
For practical purposes,
for UCDC v.1 results are printed every ten thousand iterations. \textit{Calculation of $F(x)$ is not included in running time of the methods.}
\textbf{Fig.1a} shows how the objective function $F(x)$ decreases as a function of the number of iterations. \textbf{Fig.1b} shows
how the objective function $F(x)$ decreases as a function of wall-clock time measured in seconds.}
\label{fig1}%
\end{figure}
In this figure notice that all methods were terminated after $10^4$ seconds. For practical purposes, for UCDC v.1 results are shown
every $10^4$ iterations. For all other methods results are shown after the first iteration takes place and then at every iteration. Observe in sub Figure \ref{fig1a} that block
methods RCD v.1, RCD v.2 and UCDC v.2 performed fewer iterations compared to the single coordinate UCDC v.1.
This is due to much larger per iteration computational complexity of the former methods compared to the latter.
%In Figure \ref{fig1b} it is clear that block methods RCD v.1 and v.2, despite having more expensive per iteration complexity,
%they were more efficient than both UCDC versions.
RCD v.2 despite its larger per iteration computational complexity among all methods it was the only one that solved
the problem to higher accuracy within the required maximum time. Moreover, observe in sub Figure \ref{fig1b} that for purely practical purposes it might be better to have a combination of methods
RCD v.1 and v.2. The former could be used at the beginning of the process while the latter could be used at later stages in order to guarantee robustness and speed
closer to the optimal solution. \emph{Finally, it is important to mention that on this problem for both RCD versions unit step sizes $\alpha$ were accepted by the backtracking line search for a major part of the process. Hence, backtracking line search was inexpensive.}

\subsection{$\ell_1$-Regularized Logistic Regression}
In this section we present the performance of RCD and UCDC on the $\ell_1$-regularized logistic regression problems \eqref{logistic_regression}.

Such problems are important in machine learning and are used for training a linear classifier $x\in\mathbb{R}^N$ that separates input data into two distinct clusters, for example, see \cite{yuanho} for further details.

We present the performance of the methods on two sparse large scale data sets.
Problem details are given in Table \ref{LRprobs}, where $A\in\mathbb{R}^{m\times N}$ is a matrix whose rows are training samples.
\begin{table}[h!]
	\centering
	\caption{Properties of two $\ell_1$-regularized logistic regression problems. The second and third columns show the number of training samples and features, respectively.
	The fourth column shows the sparsity of matrix $A$.}
	\begin{tabular}{lccc}
		\textbf{Problem}& \textbf{$\boldsymbol m$}	& \textbf{$\boldsymbol N$}		 & \textbf{$\mathbf{ nnz(A)/(mN)}$}  \\
		webspam	& $350,000$	& $16,609,143$	& $2.24e$-$4$    \\
		kdd2010 (algebra)	& $8,407,752$     & $20,216,830$	& $1.79e$-$6$  \\
	\end{tabular}
	\label{LRprobs}
\end{table}

The data sets can be downloaded from the collection of LSVM problems in \url{http://www.csie.ntu.edu.tw/~cjlin/libsvmtools/datasets/}.
For both experiments we set $c=10$, which resulted in more than $99\%$ classification accuracy of the used data sets.

By \cite[Table~10]{petermartin},
the coordinate Lipschitz constants for UCDC are
$$
L_j := \frac{1}{4}\sum_{q=1}^m (A_{qj}y_q)^2 \quad \forall j=1,2,\cdots,N,
$$
where $A_{qj}$ is the component of matrix $A$ at $qth$ row and $jth$ column.
For block versions RCD v.1, RCD v.2 and UCDC v.2, we set $\tau = \ceil{0.001 N}$.

The result of this experiment is shown in Figure \ref{fig2}.
\begin{figure}[h!]%
\centering
\subfloat[webspam, $F(x)$ against iterations]{%
\label{fig2a}%
\includegraphics[scale=0.32]{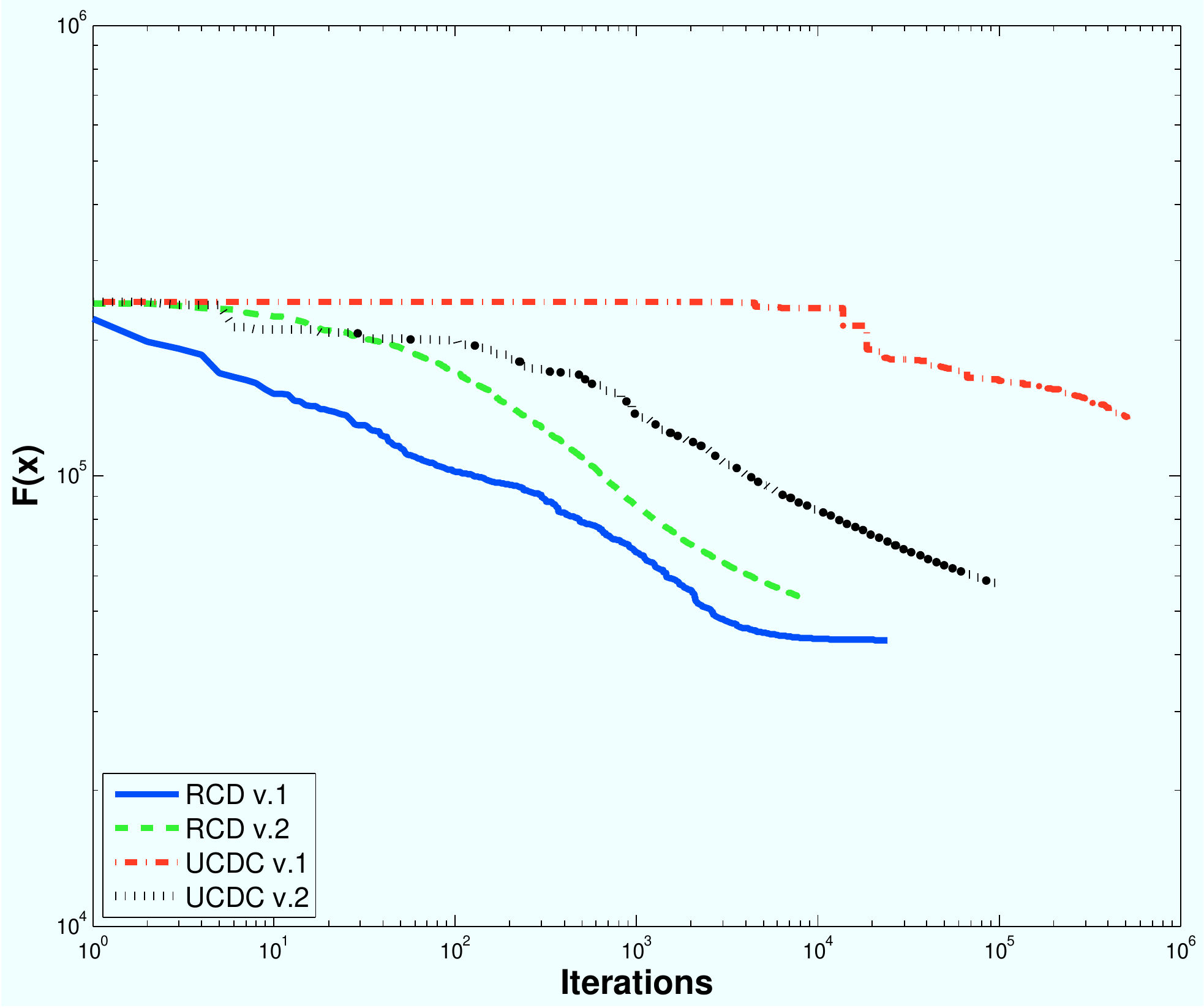}}
\subfloat[webspam, $F(x)$ against time]{%
\label{fig2b}%
\includegraphics[scale=0.32]{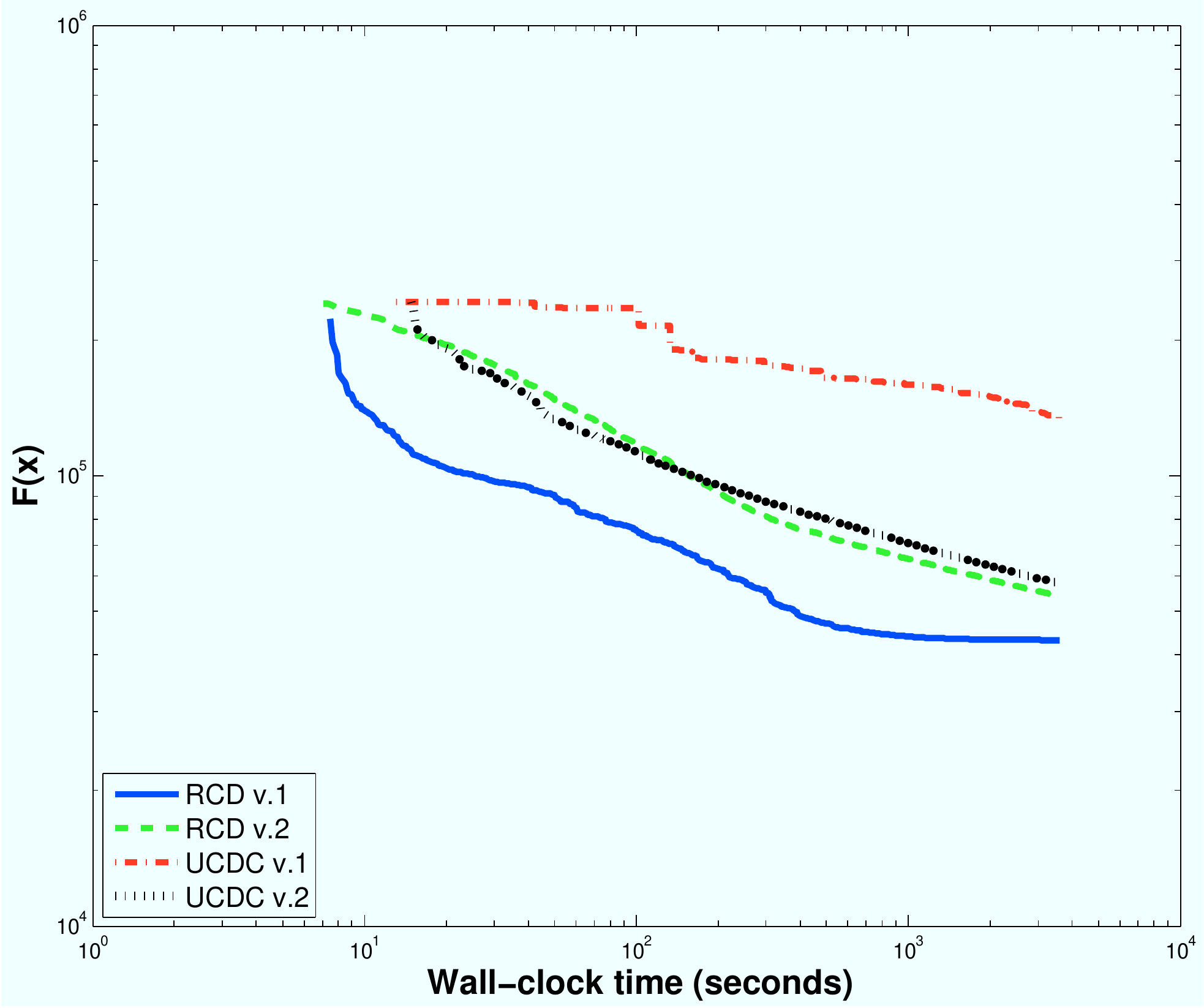}}\\
\subfloat[kdda, $F(x)$ against iterations]{%
\label{fig2c}%
\includegraphics[scale=0.32]{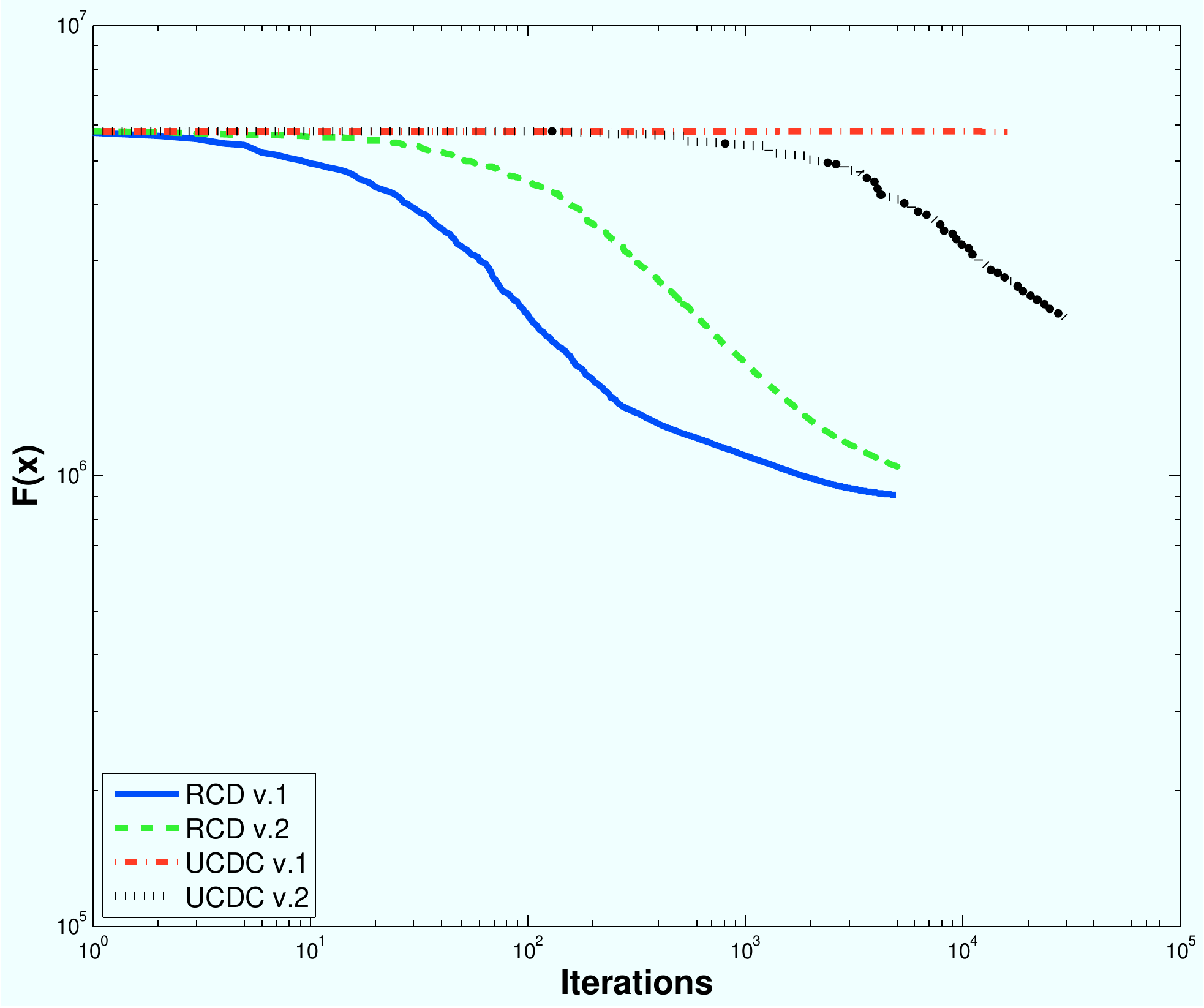}}
\subfloat[kdda, $F(x)$ against time]{%
\label{fig2d}%
\includegraphics[scale=0.32]{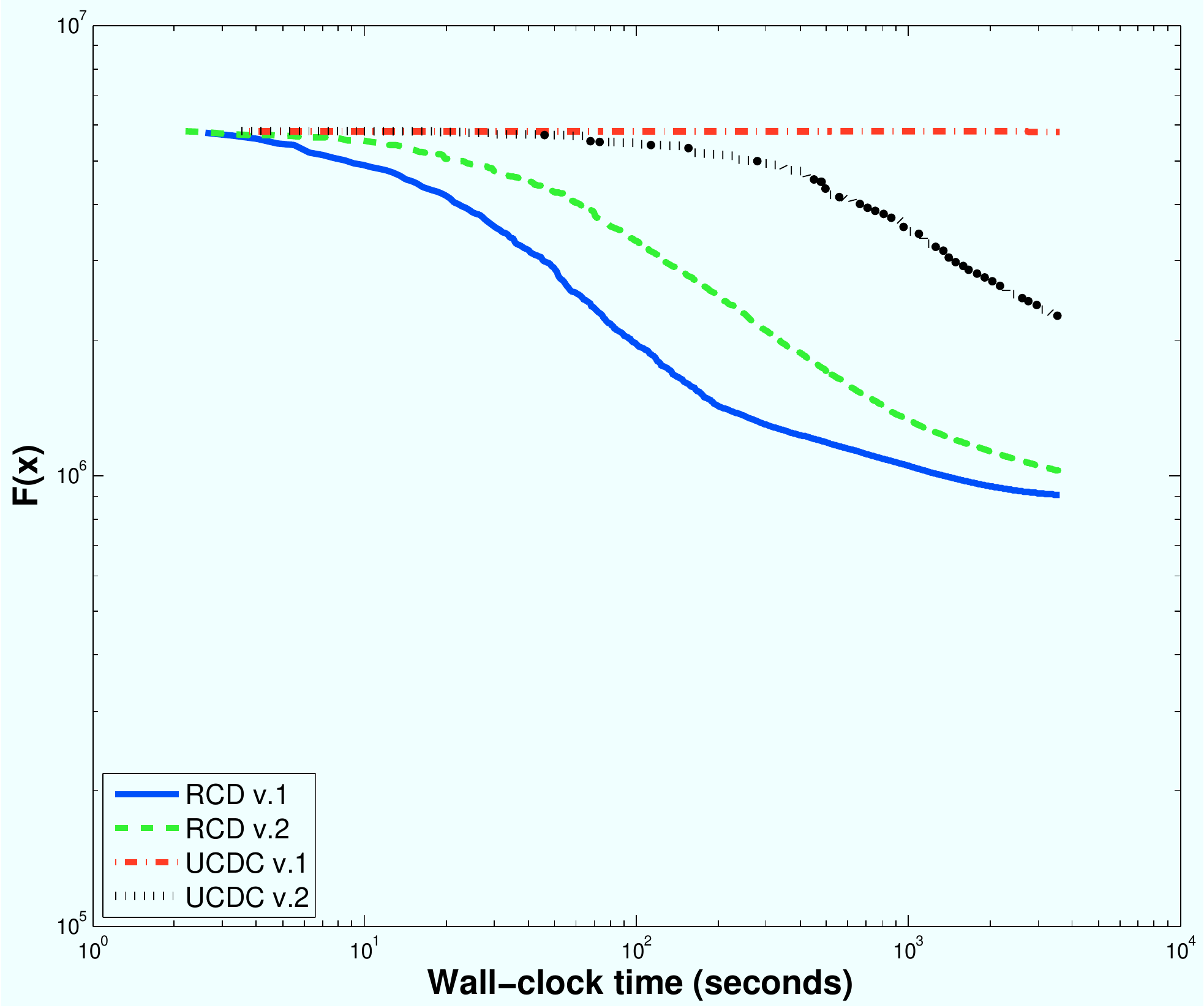}}\\
\caption{Performance of all four methods RCD v.1 and v.2 and UCDC v.1 and v.2 on two large scale $\ell_1$-regularized logistic regression problems.
The first and second rows of figures show the results for problems \textit{webspam} and \textit{kdda}, respectively.
\textit{Calculation of $F(x)$ is not included in running time of the methods.}}
\label{fig2}%
\end{figure}
In this experiment all methods were terminated after one hour of running time. Notice that RCD versions were more efficient than both UCDC versions,
with RCD v.1 being the fastest among all.  An interesting observation in Figures \ref{fig2a} and \ref{fig2c} is that RCD versions had similar per iteration computational complexity
since they performed similar number of iterations within the maximum allowed running time. However, for RCD v.1, it seems that diagonal information from the second order derivatives of $f$ was enough
to decrease faster the objective function for all iterations compared to RCD v.2. Finally, in this experiment we observed that both RCD versions accepted unit step sizes for a major part
of the process.

%Observe in sub Figure \ref{fig1a} that the block
%versions RCD v.1, RCD v.2 and UCDC v.2 performed fewer iterations compared to the single coordinate UCDC v.1.
%This is due to the much larger per iteration computational complexity of the former methods compared to the latter. Despite having more expensive
%per iteration complexity in Figure \ref{fig1b} it is clear, especially for RCD v.1 and v.2, that the block methods were more efficient than UCDC v.1.
%In fact the method RCD v.2 which has the largest per iteration computational complexity among all methods was the only one that actually solved
%the problem within the required maximum time. Finally, observe in sub Figure \ref{fig1b} that for purely practical purposes it might be better to have a combination of methods
%RCD v.1 and v.2. The former could be used at the beginning of the process while the latter could be used at later stages in order to guarantee robustness and speed
%close to the optimal solution.

\section{Conclusion}
We presented a robust randomized block coordinate descent method for composite function problems \eqref{Def_F}, which we name \textit{Robust Coordinate Descent} (RCD), that can properly handle second-order (curvature) information.
The proposed method can vary from first- to second-order; depending on how large the block updates are set, how accurate second-order information are used and how
inexactly the arising subproblems are solved.
Although the per iteration computational complexity might be higher for RCD, we present synthetic and real world large scale examples where the number of iterations substantially decreases, as well as the overall time.

From the theoretical point of view, we prove global convergence of RCD and under standard assumptions we show that RCD exhibits on expectation \textit{block} quadratic or superlinear rate of convergence.

%\pagebreak

\bibliographystyle{plain}
\bibliography{references}

\end{document}

%% file: commands.tex
\usepackage{amsfonts}
\usepackage{amsmath}
\usepackage{amssymb}

\usepackage[dvips]{graphicx}
\usepackage{color}
\usepackage{algorithmic}
\usepackage{algorithm}
\usepackage{multirow}

\usepackage{comment}
\usepackage{stackrel}
\usepackage{url}
\usepackage{subfloat}
\usepackage{subfig}

\newtheorem{assumption}[theorem]{Assumption}
\newtheorem{remark}[theorem]{Remark}

\newcommand{\R}{\mathbb{R}}

\newcommand{\eqdef}{:=}
\newcommand{\ii}{^{(i)}}

\newcommand{\Lip}{L}

\newcommand{\prox}{{\rm{prox}}}

\DeclareMathOperator*{\argmin}{arg\,min}

\usepackage{mathtools}
\DeclarePairedDelimiter{\ceil}{\lceil}{\rceil}

\def\mplus{\mathrel{%
  \ooalign{\raise.29ex\hbox{$\scriptscriptstyle\mathbf{+}$}\cr}}}

 \def\mminus{\mathrel{%
  \ooalign{\raise.29ex\hbox{$\scriptscriptstyle\mathbf{-}$}\cr}}}

%% file: RCD_June_NEW.bbl
\begin{thebibliography}{10}

\bibitem{Rockafellar06}
P.~Alart, O.~Maisonneuve, and R.~T. Rockafellar.
\newblock {\em Nonsmooth Mechanics and Analysis: Theoretical and Numerical
  Advances}.
\newblock Springer US, 2006.

\bibitem{sqa}
R.~H. Byrd, J.~Nocedal, and F.~Oztoprak.
\newblock An inexact successive quadratic approximation method for convex l-1
  regularized optimization.
\newblock Technical report, Northwestern University, September 2013.
\newblock arXiv:1309.3529v1 [math.OC].

\bibitem{Candes06}
E.~Cand\`{e}s.
\newblock Compressive sampling.
\newblock In {\em International Congress of Mathematics}, volume~3, pages
  1433--1452, Madrid, Spain, 2006.

\bibitem{IEEEhowto:CandesRombergTao}
E.~J. Cand\'{e}s, J.~Romberg, and T.~Tao.
\newblock Robust uncertainty principles: Exact signal reconstruction from
  highly incomplete frequency information.
\newblock {\em IEEE Trans. Inf. Theory}, 52(2):489--509, 2006.

\bibitem{Donoho06}
D.~Donoho.
\newblock Compressed sensing.
\newblock {\em IEEE Trans. on Information Theory}, 52(4):1289--1306, April
  2006.

\bibitem{facchinei14}
F.~Facchinei, S.~Sagratella, and G.~Scutari.
\newblock Parallel algorithms for big data optimization.
\newblock Technical report, March 2014.
\newblock arXiv:1402.5521v3 [cs.DC].

\bibitem{Fercoq13}
O.~Fercoq and P.~Richt\'{a}rik.
\newblock Accelerated, parallel and proximal coordinate descent.
\newblock Technical report, University of Edinburgh, December 2013.
\newblock arXiv:1312.5799v2 [math.OC].

\bibitem{l1regSCfg}
K.~Fountoulakis and J.~Gondzio.
\newblock A second-order method for strongly convex $\ell_1$-regularization
  problems.
\newblock Technical report, University of Edinburgh, April 2014.
\newblock arXiv:1306.5386v4 [math.OC].

\bibitem{Horn85}
R.~A. Horn and C.~R. Johnson.
\newblock {\em Matrix Analysis}.
\newblock Cambridge University Press, 1985.

\bibitem{Karimi14}
S.~Karimi and S.~Vavasis.
\newblock {IMRO}: a proximal quasi-{N}ewton method for solving
  $l_1$-regularized least square problem.
\newblock Technical report, University of Waterloo, January 2014.
\newblock arXiv:1401.4220v1 [math.OC].

\bibitem{Lee12}
J.~D. Lee, Y.~Sun, and M.~A. Saunders.
\newblock Proximal {N}ewton-type methods for convex optimization.
\newblock {\em Advances in Neural Information Processing Systems}, pages
  836--844, 2012.

\bibitem{Lee13}
J.~D. Lee, Y.~Sun, and M.~A. Saunders.
\newblock Proximal {N}ewton-type methods for minimizing composite functions.
\newblock Technical report, Stanford University, December 2013.

\bibitem{Lu13}
Z.~Lu and L.~Xiao.
\newblock On the complexity analysis of randomized block-coordinate descent
  methods.
\newblock Technical report, Simon Fraser University, May 2013.
\newblock arXiv:1305.4723v1 [math.OC].

\bibitem{Necoara14}
I.~Necoara and A.~Patrascu.
\newblock A random coordinate descent algorithm for optimization problems with
  composite objective function and linear coupled constraints.
\newblock {\em Computational Optimization and Applications}, 57(2):307--337,
  2014.

\bibitem{Nesterov04}
Yu. Nesterov.
\newblock {\em Introductory Lectures on Convex Optimization: A Basic Course}.
\newblock Applied Optimization. Kluwer Academic Publishers, 2004.

\bibitem{Nesterov12}
Yu. Nesterov.
\newblock Efficiency of coordinate descent methods on huge-scale optimization
  problems.
\newblock {\em SIAM Journal on Optimization}, 22(2):341--362, 2012.

\bibitem{nesterovgen}
Yu. Nesterov.
\newblock Gradient methods for minimizing composite functions.
\newblock {\em Mathematical Programming}, 140(1):125--161, 2013.

\bibitem{IEEEhowto:wrightbook2}
J.~Nocedal and S.~J. Wright.
\newblock {\em Numerical Optimization}.
\newblock Springer Series in Operations Research. Springer-Verlag, New York,
  1999.

\bibitem{Qin13}
Z.~Qin, K.~Scheinberg, and D.~Goldfarb.
\newblock Efficient block-coordinate descent algorithms for the group lasso.
\newblock {\em Mathematical Programming Computation}, 5(2):143--169, 2013.

\bibitem{petermartin}
P.~Richt\'{a}rik and M.~Tak\'{a}\v{c}.
\newblock Iteration complexity of randomized block-coordinate descent methods
  for minimizing a composite function.
\newblock {\em Mathematical Programming}, 2012.

\bibitem{Richtarik12}
P.~Richt\'{a}rik and M.~Tak\'{a}\v{c}.
\newblock Parallel coordinate descent methods for big data optimization.
\newblock Technical report, University of Edinburgh, December 2013.
\newblock arXiv:1212.0873v2 [math.OC].

\bibitem{desantis14}
M.~De Santis, S.~Lucidi, and F.~Rinaldi.
\newblock A fast active set block coordinate descent algorithm for
  $\ell_1$-regularized least squares.
\newblock Technical report, March 2014.
\newblock arXiv:1403.1738v2 [math.OC].

\bibitem{Scheinberg13}
K.~Scheinberg and X.~Tang.
\newblock Practical inexact proximal quasi-{N}ewton method with global
  complexity analysis.
\newblock Technical report, Lehigh University, November 2013.
\newblock arXiv:1311.6547v3 [cs.LG].

\bibitem{ShalevSchwartz13}
S.~Shalev-Schwartz and T.~Zhang.
\newblock Stochastic dual coordinate ascent methods for regularized loss
  minimization.
\newblock {\em Journal of Machine Learning Research}, 14:567--599, 2013.

\bibitem{Simon12}
N.~Simon and R.~Tibshirani.
\newblock Standardization and the group lasso penalty.
\newblock {\em Statistica Sinica}, 22(3):983, 2012.

\bibitem{Tappenden13}
R.~Tappenden, P.~Richt\'{a}rik, and J.~Gondzio.
\newblock Inexact coordinate descent: Complexity and preconditioning.
\newblock Technical report, University of Edinburgh, April 2013.
\newblock arXiv:1304.5530v1 [math.OC].

\bibitem{IEEEhowto:Tibshirani}
R.~Tibshirani.
\newblock Regression shrinkage and selection via the lasso.
\newblock {\em Journal of the Roy. Statist. Soc.}, 58(1):267--288, 1996.

\bibitem{Wright12}
S.~J. Wright.
\newblock Accelerated block-coordinate relaxation for regularized optimization.
\newblock {\em SIAM Journal of Optimization}, 22(1):159--186, 2012.

\bibitem{yuanho}
G.~X. Yuan, C.~H. Ho, and C.~J. Lin.
\newblock Recent advances of large-scale linear classification.
\newblock {\em Proceedings of the IEEE}, 100(9):2584--2603, 2012.

\end{thebibliography}
